\documentclass[12pt]{amsart}
\usepackage[T1]{fontenc}
\usepackage{amssymb,amsmath,amsthm,latexsym}
\usepackage{mathrsfs}
\usepackage{a4wide}
\usepackage[usenames,dvipsnames]{color}
\usepackage{euscript}
\usepackage{graphicx}
\usepackage{mdwlist}
\usepackage{mathtools,dsfont,wasysym}
\usepackage{stmaryrd}
\usepackage[dvipsnames]{xcolor}

\DeclareFontFamily{U}{mathx}{\hyphenchar\font45}
\DeclareFontShape{U}{mathx}{m}{n}{
      <5> <6> <7> <8> <9> <10>
      <10.95> <12> <14.4> <17.28> <20.74> <24.88>
      mathx10
      }{}
\newcommand{\abs}[1]{\vert #1\vert}
\usepackage{todonotes}

\newtheorem{theorem}{Theorem}[section]
\newtheorem*{theorema}{Theorem A}
\newtheorem*{theoremb}{Theorem B}

\newtheorem*{theoremc}{Theorem C}
\newtheorem{lemma}[theorem]{Lemma}
\newtheorem{corollary}[theorem]{Corollary}
\newtheorem{proposition}[theorem]{Proposition}
\newtheorem{fact}[theorem]{Fact}
\newtheorem*{claim*}{Claim}
\theoremstyle{remark}
\newtheorem{remark}[theorem]{Remark}
\theoremstyle{definition}
\newtheorem{definition}[theorem]{Definition}

\newtheorem{example}[theorem]{Example}
\numberwithin{equation}{section}
\makeatother

\newcommand{\nn}[1]{{\left\vert\kern-0.25ex\left\vert\kern-0.25ex\left\vert #1 
    \right\vert\kern-0.25ex\right\vert\kern-0.25ex\right\vert}}
\renewcommand{\leq}{\leqslant}
\renewcommand{\geq}{\geqslant}
\newcounter{smallromans}

\newenvironment{romanenumerate}
{\begin{list}{{\normalfont\textrm{(\roman{smallromans})}}}%
  {\usecounter{smallromans}\setlength{\itemindent}{0cm}%
   \setlength{\leftmargin}{5.5ex}\setlength{\labelwidth}{5.5ex}%
   \setlength{\topsep}{.5ex}\setlength{\partopsep}{.5ex}%
   \setlength{\itemsep}{0.1ex}}}%
{\end{list}}

\newcommand{\N}{\mathbb{N}}

\newcommand{\C}{\mathbb{C}}
\newcommand{\e}{\varepsilon}
\newcommand{\p}{\varphi}
\newcommand{\n}{\left\Vert\cdot\right\Vert}

\newcommand{\auc}{\overline{\delta}}

\newcounter{smallromansdash}

{\end{list}}

\newcounter{bigromans} 
  {\end{list}}

\begin{document}
\title[Separated sets and Auerbach systems]{Separated sets and Auerbach systems\\in Banach spaces}

\author[P.~H\'ajek]{Petr H\'ajek}
\address[P.~H\'ajek]{Department of Mathematics\\ Faculty of Electrical Engineering\\Czech Technical University in Prague\\Technick\'a 2, 166 27 Praha 6\\ Czech Republic}
\email{hajek@math.cas.cz}

\author[T.~Kania]{Tomasz Kania}
\address[T.~Kania]{Mathematical Institute\\Czech Academy of Sciences\\\v Zitn\'a 25 \\115 67 Praha 1\\Czech Republic \\ and \\ Institute of Mathematics and Computer Science,\\ Jagiellonian University,\\ {\L}ojasiewicza 6, 30-348 Krak\'{o}w, Poland}
\email{kania@math.cas.cz, tomasz.marcin.kania@gmail.com}

\author[T.~Russo]{Tommaso Russo}
\address[T.~Russo]{Department of Mathematics\\Faculty of Electrical Engineering\\Czech Technical University in Prague\\Technick\'a 2, 166 27 Praha 6\\ Czech Republic}
\email{russotom@fel.cvut.cz}

\thanks{Research of the first-named author was supported in part by OPVVV CAAS CZ.02.1.01/0.0/0.0/16$\_$019/0000778.
The second-named author acknowledges with thanks funding received from GA\v{C}R project 17-27844S; RVO 67985840 (Czech Republic). 
Research of the third-named author was supported by the project International Mobility of Researchers in CTU CZ.02.2.69/0.0/0.0/16$\_$027/0008465 and by Gruppo Nazionale per l'Analisi Matematica, la Probabilit\`a e le loro Applicazioni (GNAMPA) of Istituto Nazionale di Alta Matematica (INdAM), Italy.}

\date{\today}

\begin{abstract}
The paper elucidates the relationship between the density of a Banach space and possible sizes of Auerbach systems and  well-separated subsets of its unit sphere. For example, it is proved that for a large enough space $X$, the unit sphere $S_X$ always contains an uncountable $(1+)$-separated subset. In order to achieve this, new results concerning the existence of large Auerbach systems are established, that happen to be sharp for the class of WLD spaces. In fact, we offer the first consistent example of a non-separable WLD Banach space that contains no uncountable Auerbach system, as witnessed by a renorming of $c_0(\omega_1)$. Moreover, the following optimal results for the classes of, respectively, reflexive and super-reflexive spaces are established: the unit sphere of an infinite-dimensional reflexive space contains a~symmetrically $(1+\varepsilon)$-separated subset of any regular cardinality not exceeding the density of $X$; should the space $X$ be super-reflexive, the unit sphere of $X$ contains such a subset of cardinality equal to the density of $X$. The said problem is studied for other classes of spaces too, including WLD spaces, RNP spaces, or strictly convex ones.
\end{abstract}
\maketitle

\section{Introduction}
Over the last years, a renewed interest and a rapid progress in delineating the structure of both qualitative and quantitative properties of well-separated subsets of the unit sphere of a~Banach space have been observed. Perhaps the first spark was lit by Mercourakis and Vassiliadis \cite{mv} who have identified certain classes of compact Hausdorff spaces $K$ for which the unit sphere of the Banach space $C(K)$ of all scalar-valued continuous functions on $K$ contains an uncountable (1+)-separated set, that is, an uncountable set whose distinct elements have distances strictly greater than 1. They also asked whether an `uncountable' analogue of Kottman's theorem, or even the Elton--Odell theorem, is valid for every non-separable $C(K)$-space. \smallskip 

Kottman's theorem asserts that the unit sphere of an infinite-dimensional normed space contains an infinite (1+)-separated subset. Elton and Odell improved Kottman's theorem by showing that in such circumstances one may find $\varepsilon > 0$ and an infinite subset of the sphere that is $(1+\varepsilon)$-separated. One may thus extrapolate Mercourakis' and Vassiliadis' question to all spaces and ask for the following `uncountable' version of Kottman's theorem.\smallskip
\begin{center}\emph{Must the unit sphere of a non-separable Banach space contain\\ an~uncountable (1+)-separated subset? }\end{center}
Unfortunately, the answer to this question remains unknown. It was already observed by Elton and Odell themselves that an `uncountable' version of the Elton--Odell theorem is false (\cite[p.~109]{E-O}), as witnessed by the space $c_0(\omega_1)$. Kochanek and the second-named author answered this question affirmatively (\cite[Theorem B]{KaKo}) for the class of all non-separable $C(K)$-spaces as well as all non-separable reflexive spaces (\cite[Theorem A]{KaKo}). \smallskip 

Koszmider proved that an `uncountable' version of the Elton--Odell theorem for non-separable $C(K)$-spaces is independent of ZFC \cite{Koszmider}. Interestingly, if the unit sphere of a~$C(K)$-space contains a~$(1+\varepsilon)$-separated subset for some $\varepsilon>0$, then it also contains a~2-separated subset of the same cardinality (\cite[Theorem 1]{mv}). Very recently, C\'uth, Kurka, and Vejnar \cite{CuKuVe} improved significantly \cite[Theorem B]{KaKo} by identifying a very broad class of $C(K)$-spaces whose unit spheres contain uncountable 2-separated subsets, so in the class of $C(K)$-spaces the picture is quite full; still, a purely topological characterisation of such compact spaces is still desirable. \smallskip

In the light of the recent research of the present authors \cite{HKR-sym-sep} concerning separation in separable Banach spaces, the above-discussed problems gained new natural `symmetric' counterparts, where separation, that is the distance $\|x-y\|$, is replaced with the symmetric distance, \emph{i.e.}, $\|x\pm y\|$.  \smallskip 

The primary aim of this work is to develop results concerning constructions of as large separated subsets of the unit sphere of a Banach space as possible that are moreover symmetrically separated sets, wherever possible. Advancing the understanding that Auerbach systems can be profitably exploited to approach the above problem is important part of our contribution. Consequently, a substantial part of the paper comprises general results on the existence of Auerbach systems in Banach spaces, that we shall refer to in the last part of the Introduction. (The proofs of the results presented in the Introduction are postponed to subsequent sections, where the necessary terminology is also explained.) \smallskip

Our first result to be presented states that, for large enough Banach spaces, the main question has a positive answer indeed. In the particular case where the underlying Banach space $X$ is moreover weakly Lindel\"of determined (WLD), we are able to obtain a better relation between the density character of the space $X$ and cardinality of the $(1+)$-separated subset of the sphere. It turns out that under certain assumptions that are still rather mild, it is not possible to strengthen the results and obtain either $(1+)$-separated families of unit vectors, with cardinality larger than $\omega_1$, or uncountable $(1+\e)$-separated families. We have already mentioned that the unit sphere of $c_0(\omega_1)$ does not contain uncountable $(1+\e)$-separated subsets. Koszmider noticed that every (1+)-separated subset in the unit sphere of $c_0(\Gamma)$ has cardinality at most continuum (\cite[Proposition 4.13]{KaKo}). We optimise this result by actually decreasing the continuum to $\omega_1$, which shows that even for rather well-behaved and very large spaces the size of $(1+)$-separated subsets of the sphere may be relatively small. \smallskip 

Let us then present our first main result formally.\bigskip
\begin{theorema}${}$
\begin{romanenumerate}
\item Let $X$ be a Banach space with $w^*\text{-}{\rm dens}\, X^*>\exp_2\mathfrak{c}$. Then both $X$ and $X^*$ contain uncountable symmetrically $(1+)$-separated families of unit vectors.
\item Let $X$ be a WLD Banach space with ${\rm dens}\, X>\mathfrak{c}$. Then the unit spheres of $X$ and of $X^*$ contain uncountable symmetrically $(1+)$-separated subsets.
\item Furthermore, if $X=c_0(\Gamma)$, then every $(1+)$-separated subset of $S_{c_0(\Gamma)}$ has cardinality at most $\omega_1$.
\end{romanenumerate}
\end{theorema}
The proof of Theorem A will be presented in Section \ref{Section: combinatory}; in particular the first two parts are obtained as Corollary \ref{(1+) in large space} and \ref{(1+) in WLD}, while the last assertion is Theorem \ref{c0omega1}.\smallskip

It was brought to the attention of the second-named author by Marek C\'uth quite rightly that the proof of \cite[Theorem A(iii)]{KaKo} contains a gap, namely that it is not clear why is the system constructed in the proof of \cite[Theorem 3.8]{KaKo} biorthogonal---we shall return to this point soon. However, let us notice that the second clause of Theorem A provides a remedy to this problem, together with an improvement of the result, by exhibiting the sought (1+)-separated subset both in $S_X$, and in $S_{X^*}$.\smallskip

We next turn our attention to some strong structural constrains on the space, which allow construction of potentially larger separated subsets of the unit sphere. For example, we strengthen considerably \cite[Theorem A(i)]{KaKo} by proving the existence of a symmetrically (1+)-separated set in the unit sphere of every (quasi-)reflexive space $X$ that has the maximal possible cardinality, that is, the cardinality equal to ${\rm dens}\, X$, the density of the underlying Banach space. \smallskip

In the case where the number ${\rm dens}\, X$ has uncountable cofinality, such set can be taken to be symmetrically ($1+\varepsilon$)-separated for some $\varepsilon >0$. When $X$ is a super-reflexive space, we improve \cite[Theorem A(ii)]{KaKo} by exhibiting a symmetrically ($1+\varepsilon$)-separated set in the unit sphere of $X$ that also has the maximal possible cardinality---this answers a question raised by the second-named author and T.~Kochanek, \cite[Remark 3.7]{KaKo}. Let us then present our result formally.
\begin{theoremb} Let $X$ be an infinite-dimensional, (quasi-)reflexive Banach space. Then,
\begin{romanenumerate}
\item $S_X$ contains a symmetrically $(1+)$-separated subset with cardinality ${\rm dens}\, X$;
\item for every cardinal number $\kappa\leq{\rm dens}\, X$ with uncountable cofinality there exist $\e>0$ and a symmetrically $(1+\e)$-separated subset of $S_X$ of cardinality $\kappa$;
\item if $X$ is super-reflexive, there exist $\e>0$ and a symmetrically $(1+\e)$-separated subset of $S_X$ of cardinality ${\rm dens}\, X$.
\end{romanenumerate}
\end{theoremb}
Quite remarkably, Theorem B involves properties that are preserved by isomorphisms, not only isometries, of Banach spaces.

Let us also note that clause (ii) of Theorem B is optimal as witnessed by the $\ell_2$-sum of the spaces $\ell_{p_n}(\omega_n)$, where $(p_n)_{n=1}^\infty$ is a sequence of numbers greater than 1 and increasing to $\infty$. Such (reflexive) Banach space has density $\omega_\omega$, yet its unit sphere does not have a $(1+\e)$-separated subset of cardinality $\omega_\omega$ for any $\e > 0$ (see \cite[Remark 3.7]{KaKo} for more details).\smallskip

A Banach space $X$ is \emph{quasi-reflexive} whenever the canonical image of $X$ in its bidual $X^{**}$ has finite codimension. The proof of Theorem B will be presented in Section \ref{Section: geometry}; in particular, the first clause is Theorem \ref{(1+) in reflexive}, the second one is contained in Corollary \ref{(1+e) in reflexive} and the last part will be discussed in Section \ref{injections lp(gamma)}. In the same chapter, we shall also present further results whose proofs follow similar patterns; in particular, we shall also prove a~generalisation of the first two parts of the theorem to the class of Banach spaces with the RNP. \smallskip 

Loosely speaking, the very idea behind the proofs of clauses (i) and (ii) of Theorem A is based on employing `very long' Auerbach systems to extract uncountable $(1+)$-separated subsets. Therefore, a substantial part of the present contribution comprises quite general results concerning Auerbach systems in Banach spaces. To wit, we prove in the first part of Section \ref{Section: combinatory} that a sufficiently large Banach space contains a large Auerbach system, which is probably of interest on its own. We next improve and optimise this assertion, in the case of a WLD Banach space. Let us give the statement of our result, prior to any further comment.
\begin{theoremc} Let $X$ be a WLD Banach space.
\begin{romanenumerate}
\item Suppose that ${\rm dens}\, X > \omega_1$. Then, $X$ contains a subspace $Y$ with Auerbach basis and such that ${\rm dens}\,Y={\rm dens}\,X$.
\item \emph{(CH)} There exists a renorming $\nn\cdot$ of the real space $c_0(\omega_1)$ such that the space $(c_0(\omega_1),\nn\cdot)$ contains no uncountable Auerbach systems.
\end{romanenumerate}
\end{theoremc}

Clause (ii) of Theorem C is perhaps the most striking single result of the whole paper, and it demonstrates, in particular, that part (i) can not be improved, as the result is consistently false when ${\rm dens}\,X=\omega_1$. The entire Section \ref{No uncountable Auerbach} will be dedicated to the proof of the said result.\smallskip

According to a celebrated result due to Kunen, there exists (under the assumption of the Continuum Hypothesis) a non-separable Banach space with virtually no (uncountable) system of coordinates, in that it admits no uncountable biorthogonal system. In a sense, we may also view Theorem C(ii) as a counterpart to Kunen's result within the class of WLD Banach spaces. Obviously, every non-separable WLD Banach space admits a biorthogonal system with the maximal possible cardinality (any bounded M-basis is a witness of this), so there is no Kunen-type example in the context of WLD spaces.\smallskip

The second clause should also be compared to some results by Godun, Lin, and Troyanski (\cite{GLT Auerbach}, see also \cite[Section 4.6]{HMVZ}), who proved that if $X$ is a non-separable Banach space such that $B_{X^*}$ is $w^*$-separable, then there exists an equivalent norm $\nn\cdot$ on $X$ such that $(X,\nn\cdot)$ admits no Auerbach basis. The result applies, \emph{e.g.}, to $\ell_1([0,1])$, for which space the claim was already proved in \cite{Godun2}; also see \cite{Godun1}. In the same direction, we should also mention that Plichko (\cite{Plichko}) proved, in particular, that $c_0[0,1]+C[0,1]\subseteq\ell_\infty[0,1]$ has no Auerbach basis in its canonical norm.

The examples in \cite{Godun1, Godun2} exhibited examples of Banach spaces with unconditional bases, but no Auerbach basis; this motivated the authors of \cite{GMZ open} to pose the question of whether there exists a Banach space with unconditional basis and whose no non-separable subspace admits an Auerbach basis (\cite[Problem 294]{GMZ open}). Therefore, (ii) of the previous theorem also provides an answer to a stronger version of this question, at least under the assumption of the Continuum Hypothesis.

\section{Notation and preliminaries}
Our notation is standard, as in most textbooks in Functional Analysis. For any unexplained notation or definition, such as super-reflexivity, or the Radon--Nikodym property (RNP, for short), we refer, \emph{e.g.}, to \cite{ak,BeLi,diestel,FHHMZ,LiTzaI,LiTzaII}. Most our results are valid for Banach spaces over either the real or complex field. Any result in which the scalar field is not explicitly mentioned is understood to apply to both cases. For a normed space $X$, we shall denote by $S_X$ the unit sphere of $X$ and by $B_X$ the closed unit ball of $X$.\smallskip

Let us present here the formal definitions of the separation notions relevant to this paper. 
\begin{definition} Let $\delta > 0$ be given. We say that a subset $A$ of a normed space is  
\begin{itemize}
\item \emph{$(\delta+)$-separated} (respectively, \emph{$\delta$-separated}) whenever we have $\|x-y\|>\delta$ (re\-spe\-cti\-vely, $\|x-y\|\geq\delta$) for any distinct elements $x,y\in A$;
\item \emph{symmetrically $(\delta+)$-separated} (respectively, \emph{symmetrically $\delta$-separated}) whenever $\|x\pm y\| > \delta$ (respectively, $\|x\pm y\|\geqslant \delta$) for any distinct elements $x,y\in A$.
\end{itemize}\end{definition}

\subsection{Exposed and strongly exposed points}\label{review exposed} Let $X$ be a Banach space and let $C\subseteq X$ be a non-empty, closed, convex and bounded set. A point $x\in C$ is an \emph{exposed point} for $C$ if there is a functional $\p \in X^*$ such that ${\rm Re}\,\langle\p,y\rangle<{\rm Re}\,\langle\p,x\rangle$ for every $y\in C$, $y\neq x$. In other words, ${\rm Re}\,\p$ attains its supremum over $C$ at the point $x$ and only at that point. In such a case, we also say that the functional $\p$ \emph{exposes} the point $x$. $x\in C$ is a \emph{strongly exposed point} for $C$ if there is a functional $\p \in X^*$ that exposes $x$ and with the property that $y_n\rightarrow x$ for every sequence $(y_n)_{n=1}^\infty$ in $C$ such that $\langle\p,y_n\rangle \rightarrow\langle\p,x\rangle$. In such a case, we say that $\p$ \emph{strongly exposes} the point $x$.\smallskip

Of course, every strongly exposed point is an exposed point and it is immediate to check that every exposed point is an extreme point. By a result of Lindenstrauss and Troyanski (\cite{Lind exposed,Troyanski} see, \emph{e.g.}, \cite[Theorem 8.13]{FHHMZ}), every convex, weakly compact set in a Banach space is the closed convex hull of its strongly exposed points. We shall use the immediate consequence that every non-empty, convex, and weakly compact set in a Banach space admits an exposed point.\smallskip

A functional $\p\in S_{X^*}$ is a \emph{$w^*$-exposed point} of $B_{X^*}$, whenever there exists a unit vector $x\in S_X$ such that $\langle\p,x\rangle=1$ and ${\rm Re}\,\langle\psi,x\rangle<1$ for every $\psi\in B_{X^*}$, $\psi\neq\p$; in other words, $\p$ is the unique supporting functional at $x$. A Banach space $X$ is called a \emph{G\^ateaux differentiability space} if every convex continuous function defined on a~non-empty open convex subset $D$ of $X$ is G\^ateaux differentiable at densely many points of $D$. This notion differs from the notion of \emph{weak Asplund space} only by virtue of the fact that the set of differentiability points is not required to contain a dense $G_\delta$, but merely to be dense in $D$ (for information concerning those spaces, consult \cite{Fabian w-asplund,Phelps diff}). On the other hand, let us stress the fact that the two notions are actually distinct, as it was first proved in \cite{MoorsSomasundaram}. Let us also refer to \cite{Moors} for a simplified proof of the example and to \cite{Kalenda Stegall, Kalenda wAsplund, KenMooSci, MooSom} for related results. 

The interplay between $w^*$-exposed points and G\^ateaux differentiability spaces stems from the fact that points of G\^ateaux differentiability of the norm correspond to $w^*$-exposed points in the dual space. More precisely, the norm $\n$ of a~Banach space $X$ is G\^ateaux differentiable at $x\in S_X$ if and only if there exists a unique $\p\in B_{X^*}$ with $\langle\p,x\rangle=1$ (see, \emph{e.g.}, \cite[Corollary 7.22]{FHHMZ}); in which case, $\p$ is $w^*$-exposed by $x$. \smallskip

\subsection{Asymptotically uniformly convex spaces}\label{review AUC} \begin{definition} Let $X$ be an infinite-dimensional Banach space. The \emph{modulus of asymptotic uniform convexity} $\auc_X$ is given, for $t\geq0$, by
$$\auc_X(t):=\inf_{\|x\|=1}\sup_{\dim(X/H)<\infty}
\inf_{\overset{h\in H}{\|h\|\geq t}}(\|x+h\|-1).$$
We shall also denote by $\auc_X(\cdot,x)$ the \emph{modulus of asymptotic uniform convexity at $x$}:
$$\auc_X(t,x):=\sup_{\dim(X/H)<\infty}\inf_{\overset{h\in H}{\|h\|\geq t}}(\|x+h\|-1).$$
An infinite-dimensional Banach space $X$ is \emph{asymptotically uniformly convex} if $\auc_X(t)>0$ for every $t>0$.
\end{definition}

Let us first note that $\auc_X(t,x)\geq0$ for every $t\geq0$ and $x\in S_X$. In fact, we can find a norming functional $x^*$ for $x$ and consider $H=\ker x^*$; of course, $\|x+h\|\geq1$ for each $h\in H$. Hence
$$\auc_X(t,x)\geq\inf_{\overset{h\in \ker x^*}{\|h\|\geq t}}(\|x+h\|-1)\geq0.$$
Choosing, for each $H$ with $\dim(X/H)<\infty$, $h=0\in H$, shows that $\auc_X(0,x)=0$, so $\auc_X(0)=0$ too. It is also obvious that $\auc_X$ is a non-decreasing function, and that $\auc_X(\cdot,x)$ is non-decreasing for each fixed $x$; in fact, $\inf_{h\in H,\|h\|\geq t}(\|x+h\|-1)$ clearly increases with $t$.

One more property which is easily verified (\emph{cf.} \cite[Proposition 2.3.(3)]{JLPS diff Lip}) is the fact that, for each $t\in[0,1]$, $\delta_X(t)\leq\auc_X(t)$, where $\delta_X$ denotes the modulus of uniform convexity. The non-unexpected fact that uniformly convex Banach spaces are asymptotically uniformly convex immediately follows. The two notions are however non equivalent, as it is easy to see that $\ell_1$ is asymptotically uniformly convex. For more information on asymptotically uniformly convex Banach spaces and the modulus $\auc_X$, consult, \emph{e.g.}, \cite{GKL, JLPS diff Lip, KOS asymptotic, milman}.\smallskip

In the proof of one our result we shall need one more property of this modulus, namely that passage to a subspace improves the modulus $\auc$ (\emph{cf.} \cite[Proposition 2.3.(2)]{JLPS diff Lip}). We also present its very simple proof, for the sake of completeness.
\begin{fact}\label{Fact: AUC modulus of subspace} Let $Y$ be a closed infinite-dimensional subspace of a Banach space $X$. Then $\auc_Y\geq\auc_X$.
\end{fact}
\begin{proof} Fix any $t\geq0$ and $y\in Y$. If $H\subseteq X$ is such that $\dim(X/H)<\infty$, then also $\dim(Y/(Y\cap H))<\infty$ (the inclusion $Y\hookrightarrow X$ induces an injection $Y/(Y\cap H)\hookrightarrow X/H$); consequently, 
$$\inf_{\overset{h\in H}{\|h\|\geq t}}(\|y+h\|-1)\leq \inf_{\overset{h\in H\cap Y}{\|h\|\geq t}}(\|y+h\|-1) \leq\overline{\delta}_Y(t,y).$$
Passage to the supremum over $H$ gives $\auc_X(t)\leq\auc_X(t,y)\leq\auc_Y(t,y)$. We now pass to the infimum over $y\in Y$ and conclude the proof.
\end{proof}

\subsection{WLD spaces}\label{Review WLD} A topological space $K$ is a \emph{Corson compact} whenever it is homeomorphic to a compact subset $C$ of the product space $[-1,1]^\Gamma$ for some set $\Gamma$, such that every element of $C$ has only countably many non-zero coordinates. A Banach space $X$ is \emph{weakly Lindel\"of determined} (hereinafter, \emph{WLD}) if the dual ball $B_{X^*}$ is a Corson compact in the relative $w^*$-topology.\smallskip

Suppose now that $\{x_\gamma;x_\gamma^*\}_{\gamma\in\Gamma}\subseteq X\times X^*$ is an \emph{M-basis} for $X$, namely 
\begin{romanenumerate}
\item $\overline{\rm span}\{x_\gamma\}_{\gamma\in\Gamma}=X$;
\item $\overline{\rm span}^{w^*}\{x_\gamma^*\}_{\gamma\in\Gamma}=X^*$;
\item $\langle x_\alpha^*, x_\beta\rangle =\delta_{\alpha,\beta}$.
\end{romanenumerate}
Given an M-basis $\{x_\gamma;x_\gamma^*\}_{\gamma\in\Gamma}$, a functional $x^*\in X^*$ is \emph{countably supported} by $\{x_\gamma;x_\gamma^*\}_{\gamma \in\Gamma}$, or $\{x_\gamma;x_\gamma^*\}_{\gamma \in\Gamma}$ \emph{countably supports} $x^*$, if the support of $x^*$, \emph{i.e.},
$${\rm supp}\, x^*:=\{\gamma\in\Gamma\colon \langle x^*, x_\gamma\rangle \neq0\}$$
is a countable subset of $\Gamma$. 

We also say that an M-basis $\{x_\gamma;x_\gamma^*\}_{\gamma\in\Gamma}$ is \emph{countably $1$-norming} if the set of countably supported functionals is $1$-norming for $X$, \emph{i.e.}, every $x\in X$ satisfies
$$\|x\|=\sup\{|\langle x^*, x\rangle|\colon x^*\in B_{X^*}\text{ is countably supported}\}.$$

We may now give the following M-bases characterisation of WLD Banach spaces (\cite{Kalenda,Kalenda survey,VWZ}, \emph{cf.} \cite[Theorems 5.37 and 5.51]{HMVZ}).
\begin{theorem}
Let $X$ be a Banach space. Then the following assertions are equivalent:
\begin{romanenumerate}
\item $X$ is WLD;
\item $X$ admits an M-basis $\{x_\gamma;x_\gamma^*\}_{\gamma\in\Gamma}$ that countably supports $X^*$, i.e., every $x^*\in X^*$ is countably supported by $\{x_\gamma; x_\gamma^*\}_{\gamma\in\Gamma};$
\item $X$ admits, under any equivalent norm, a countably $1$-norming M-basis. 
\end{romanenumerate}
In this case, every M-basis countably supports $X^*$.
\end{theorem}

For a more detailed presentation of these notions, the reader may wish to consult \cite{HMVZ,KKLP,Kalenda survey} and the references therein.

\subsection{Infinitary combinatorics}\label{Review combinatory} In subsequent sections we shall need to exploit some results concerning infinitary combinatorics, whose statements are recalled here for convenience of the reader. We also shortly fix some notation concerning cardinal and ordinal numbers.\smallskip

We use von Neumann's definition of ordinal numbers and we regard cardinal numbers as initial ordinal numbers. In particular, we write $\omega$ for $\aleph_0$, $\omega_1$ for $\aleph_1$, \emph{etc}., as we often view cardinal numbers as well-ordered sets; we also denote by $\mathfrak{c}$ the cardinality of continuum. For a cardinal number $\kappa$, we write $\kappa^+$ for the immediate successor of $\kappa$, that is, the smallest cardinal number that is strictly greater than $\kappa$. \smallskip

If $F$ and $G$ are subsets of a certain ordinal number $\lambda$, we shall use the (perhaps self-explanatory) notation $F<G$ to mean that $\sup F<\min G$; in the case that $G=\{g\}$ is a~singleton, we shall write $F<g$ instead of $F<\{g\}$.\smallskip

We next recall a few results to be used later. 
\begin{lemma}[$\Delta$-system lemma] Consider a family $\mathcal{F}=\{F_\gamma\} _{\gamma\in\Gamma}$ of finite subsets of a set $S$, where $|\Gamma|$ is an uncountable regular cardinal number. Then there exist a subset $\Gamma_0$ of $\Gamma$ with $|\Gamma_0|=|\Gamma|$ and a finite subset $\Delta$ of $S$ such that 
$$F_{\gamma}\cap F_{\gamma'}=\Delta,$$
whenever $\gamma,\gamma'\in\Gamma_0$, $\gamma\neq\gamma'$.
\end{lemma}
This result is more frequently stated in a slightly different way, \emph{i.e.}, involving a set $\mathcal{F}$, whose cardinality is a regular cardinal number; we prefer this (equivalent) formulation, since it will apply more directly to our considerations. The difference relies in the fact that, for $\gamma\neq\gamma'$, the sets $F_{\gamma}$ and $F_{\gamma'}$ are not necessarily distinct, so the cardinality of the set $\mathcal{F}$ may be smaller than $|\Gamma|$. We refer, \emph{e.g.}, to \cite[Lemma III.2.6]{kunen} for the more usual statement of the result and to the remarks following the proof of Theorem III.2.8 for a~comparison between the two formulations.\smallskip

The second result that we shall use concerns partition properties for cardinal numbers and it is the counterpart for larger cardinals to the classical Ramsey's theorem. For the proof of the result to be presented and for a more complete discussion over partition properties, we refer, \emph{e.g.}, to \cite[Theorem 5.67]{HMVZ}, \cite[Section 9.1]{Jech}, \cite[pp. 237--238]{kunen}, or the monograph \cite{EHMR}. Before we state the result, we require a piece of notation.\smallskip

For a cardinal number $\kappa$, one defines the iterated powers by $\exp_1\kappa:=2^\kappa$ and then recursively $\exp_{n+1}\kappa:=\exp(\exp_n\kappa)$ ($n\in \mathbb N$). If $S$ is a set and $\kappa$ is a cardinal number, then we denote by $[S]^\kappa$ the set of all subsets of $S$ of cardinality $\kappa$, \emph{i.e.},
$$[S]^\kappa:=\{A\subseteq S\colon |A|=\kappa\}.$$

We also need to recall the arrow notation: assume that $\kappa$, $\lambda$, and $\sigma$ are cardinal numbers and $n$ is a natural number. Then the symbol
$$\kappa\to(\lambda)^n_\sigma$$
abbreviates the following partition property: for every function $f\colon[\kappa]^n\to\sigma$ there exists a set $Z\subseteq\kappa$ with $|Z|=\lambda$ such that $f$ is constant on $[Z]^n$; in this case, we say that $Z$ is \emph{homogeneous} for $f$. With a more suggestive notation, the function $f$ is sometimes called a \emph{$\sigma$-colouring} of $[\kappa]^n$ and, accordingly, the set $Z$ is also said to be \emph{monochromatic}.
\begin{theorem}[Erd\H{o}s--Rado theorem] For every infinite cardinal $\kappa$ and every $n\in\N$
$$(\exp_n\kappa)^+\to(\kappa^+)^{n+1}_\kappa.$$
\end{theorem}

The last result of combinatorial nature recorded in this section is Hajnal's theorem on free sets. Given a set $S$, by a \emph{set function on} $S$ we understand a function $f\colon S\to 2^S$. A subset $H$ of $S$ is a \emph{free set for} $f$ if $f(x)\cap H\subseteq\{x\}$ whenever $x\in H$. One may equivalently require, and such approach is frequently followed, that the function $f$ satisfies $x\notin f(x)$ ($x\in S$), in which case $H$ is a free set for $f$ if it is disjoint from $f(x)$, for every $x\in H$. \smallskip

The natural question which arises is to find sufficient conditions on $f\colon\kappa\to 2^\kappa$, where $\kappa$ is a cardinal number, for the existence of as large as possible free sets. It is clear that the mere assumption the cardinality of $f(x)$ to be less than $\kappa$, for $x\in\kappa$, does not even ensure the existence of a two element free set. This is simply witnessed by the function $f(\lambda):=\lambda\;(=\{\alpha\colon\alpha<\lambda\})$ ($\lambda<	\kappa$). On the other hand, the existence of $\lambda<\kappa$ such that $|f(x)|<\lambda$ for $x\in\kappa$ turns out to be sufficient for the existence of free sets of the maximal possible cardinality. This result was first conjectured by Ruziewicz \cite{Ruz36} and finally proved by Hajnal \cite{Haj61}. \smallskip

Let us now formally state Hajnal's theorem; we shall refer, \emph{e.g.}, to \cite[\S 44]{EHMR} or \cite[\S 3.1]{Wil77} for the proof of the result and for further information on the subject.
\begin{theorem}[Hajnal's theorem]\label{Hajnal th} Let $\lambda$ and $\kappa$ be cardinal numbers with $\lambda<\kappa$ and $\kappa$ infinite. Then for every function $f\colon\kappa\to [\kappa]^{<\lambda}$ there exists a set of cardinality $\kappa$ that is free for $f$.
\end{theorem}

\section{Combinatorial analysis}\label{Section: combinatory}
\subsection{The r\^{o}le of Auerbach systems}\label{Section: Auerbach}
One of the main goals of the section is to see how to use Auerbach systems or, more generally, biorthogonal systems for the construction of separated families of unit vectors. Therefore, the first part of the section is dedicated to some general results about the existence of Auerbach systems. As a matter of fact, our arguments will actually produce subspaces with Auerbach bases, not merely Auerbach systems.

\begin{theorem} Let $\kappa\geq\mathfrak{c}$ be a cardinal number and let $X$ be a Banach space with $w^*\text{-}{\rm dens}\, X^*>\exp_2\kappa$. Then $X$ contains a subspace $Y$ with Auerbach basis and such that ${\rm dens}\, Y=\kappa^+$.
\end{theorem}

\begin{proof} Let $\kappa\geq\mathfrak{c}$ be a cardinal number.  Suppose that $X$ is a Banach space with the property that $\lambda:=w^*\text{-}{\rm dens}\, X^*> \exp_2\kappa$. We may then find a long basic sequence $(e_\alpha)_{\alpha<\lambda}$ of unit vectors in $X$ and a long sequence $(\p_{\alpha,\beta})_{\alpha<\beta<\lambda}$ of unit functionals in $X^*$ with the following properties:
\begin{romanenumerate}
\item $\p_{\alpha,\beta}$ is a norming functional for the molecule $e_\alpha-e_\beta$ for each $\alpha<\beta<\lambda$;
\item $e_\gamma\in\ker\p_{\alpha,\beta}$ for every $\alpha<\beta<\gamma<\lambda$.
\end{romanenumerate}
The existence of such sequences is proved by a very simple transfinite induction argument, which is a minor modification over the Mazur technique (\cite[Corollary 4.11]{HMVZ}). Assuming that we have already constructed elements $(e_\alpha)_{\alpha<\gamma}$ and $(\p_{\alpha,\beta})_{\alpha<\beta<\gamma}$ satisfying  the two properties above (for some $\gamma<\lambda$), the unique difference is that we additionally require $e_\gamma\in\cap_{\alpha<\beta<\gamma} \ker\p_{\alpha,\beta}$. This is indeed possible, as $|\{\p_{\alpha,\beta}\}_ {\alpha<\beta<\gamma}|\leq|\gamma|<\lambda$, whence the family $\{\p_{\alpha,\beta}\}_{\alpha<\beta<\gamma}$ does not separate points in $X$. In order to conclude the inductive argument, it is then sufficient to choose, for each $\alpha<\gamma$, a norming functional $\p_{\alpha,\gamma}$ for the vector $e_\alpha-e_\gamma$.\smallskip

We are now in position to invoke the Erd\H{o}s--Rado theorem; let us consider the following colouring $c\colon[\lambda]^3\to[-1,1]$ of $[\lambda]^3$. Given any $p\in[\lambda]^3$, we may uniquely write $p=\{\alpha,\beta,\gamma\}$ with $\alpha<\beta<\gamma$; we can therefore unambiguously set $c(p):=\langle\p_{\beta,\gamma},e_\alpha\rangle \in\C$. According to the Erd\H{o}s--Rado theorem, we have $(\exp_2\kappa)^+\to(\kappa^+)^3_\kappa$, whence from our assumptions we deduce \emph{a fortiori} $\lambda\to(\kappa^+)^3_\mathfrak{c}$. Consequently, there exists a subset $\Lambda$ of $\lambda$ with $|\Lambda|=\kappa^+$ which is monochromatic for the colouring $c$. In other words, there exists $t\in\C$ such that $\langle\p_{\beta,\gamma},e_\alpha\rangle=t$ for every triple $\alpha,\beta,\gamma\in\Lambda$ with $\alpha<\beta<\gamma$; of course, we immediately deduce that $e_\alpha-e_\beta\in\ker\p_{\gamma,\eta}$ whenever $\alpha,\beta,\gamma,\eta\in\Lambda$ satisfy $\alpha<\beta<\gamma<\eta$.\smallskip

In order to conclude the argument, consider the unit vectors $u_{\alpha,\beta}:=\frac{e_\alpha-e_\beta}{\|e_\alpha-e_\beta\|}$ ($\alpha,\beta\in\Lambda$, $\alpha<\beta$); those are actually well defined, since $e_\alpha\neq e_\beta$ for $\alpha\neq\beta$. From our argument above, we have $\langle\p_{\gamma,\eta},u_{\alpha,\beta}\rangle=0$ whenever $\alpha<\beta<\gamma<\eta$ are in $\Lambda$. Moreover, condition (i) clearly gives $\langle\p_{\alpha,\beta},u_{\alpha,\beta}\rangle=1$, while (ii) assures us that $\langle\p_{\alpha,\beta},u_{\gamma,\eta}\rangle=0$ ($\alpha,\beta,\gamma,\eta\in\Lambda$, $\alpha<\beta<\gamma<\eta$). At this stage it is, of course, immediate to construct an Auerbach system of length $\kappa^+$. To wit, find ordinal numbers $(\alpha_\theta)_{\theta<\kappa^+}$ and $(\beta_\theta)_{\theta<\kappa^+}$ in $\Lambda$ such that:
\begin{romanenumerate}
\item $\alpha_\theta<\beta_\theta<\alpha_\eta$ whenever $\theta<\eta<\kappa^+$;
\item $\Lambda=\{\alpha_\theta\}_{\theta<\kappa^+}\cup\{\beta_\theta\}_{\theta<\kappa^+}$.
\end{romanenumerate}
Then, the unit vectors $u_\theta:=u_{\alpha_\theta,\beta_\theta}$ ($\theta<\kappa^+$) with the corresponding biorthogonal functionals $\p_\theta:=\p_{\alpha_\theta,\beta_\theta}$ ($\theta<\kappa^+$) clearly constitute an Auerbach system. Finally, $\{u_\theta; \p_\theta\}_{\theta<\kappa^+}$ is an M-basis for the subspace $Y:=\overline{{\rm span}}\{u_\theta\}_{\theta< \kappa^+}$, which concludes the proof.
\end{proof}

It is a standard fact that the weak* density of the dual of $X$, $w^*\text{-dens}\, X^*$ does not exceed some cardinal number $\lambda$ if and only if there exists a linear continuous injection of $X$ into $\ell_\infty(\lambda)$. Consequently, when $w^*\text{-dens}\, X^*\leq\lambda$ we deduce that ${\rm dens}\, X\leq|\ell_\infty(\lambda)|=\exp\lambda$, or, in other words, ${\rm dens}\, X\leq \exp({w^*\text{-dens}\, X^*})$. Combining this inequality with the content of the previous theorem leads us to the following corollary.
\begin{corollary} Let $\kappa\geq\mathfrak{c}$ be a cardinal number. Suppose that $X$ is a Banach space with ${\rm dens}\, X>\exp_3\kappa$. Then $X$ contains a subspace $Y$ with an Auerbach basis and such that ${\rm dens}\, Y=\kappa^+$.
\end{corollary}

In the case where $X$ is a WLD Banach space, we can improve the previous result and obtain, under some cardinality assumptions on ${\rm dens}\, X$, the existence of subspaces with Auerbach bases and of the maximal possible density character, namely equal to the density character of the Banach space $X$. We also point out that the restrictions on ${\rm dens}\, X$ are in fact necessary, in view of the main result of Section \ref{No uncountable Auerbach}. In Section \ref{Review WLD} we shortly reviewed the results concerning WLD Banach spaces to be used presently; let us also refer to Section \ref{Review combinatory} for information on Hajnal's theorem.

\begin{theorem} Every WLD Banach space $X$ with ${\rm dens}\,X>\omega_1$ contains a subspace $Y$ with Auerbach basis and such that ${\rm dens}\,Y={\rm dens}\,X$.
\end{theorem}
\begin{proof} Let us denote by $\kappa={\rm dens}\,X$ and select an M-basis $\{e_\alpha;e_\alpha^*\}_{\alpha<\kappa}$ for $X$; we may assume that $\|e_\alpha\|=1$ ($\alpha<\kappa$). We may also find, for each $\alpha<\kappa$, a functional $x_\alpha^*\in S_{X^*}$ such that $\langle x_\alpha^*,e_\alpha\rangle=1$. According to the fact that $\{e_\alpha\}_{\alpha<\kappa}$ countably supports $X^*$, the sets
$$N_\alpha:={\rm supp}\,x_\alpha^*=\{\beta<\kappa\colon \langle x_\alpha^*,e_\beta \rangle\neq0\}\qquad(\alpha<\kappa)$$
are at most countable. We may therefore apply Hajnal's theorem to the function $f\colon\kappa\to[\kappa]^{<\omega_1}$ defined by $\alpha\mapsto N_\alpha$; this yields the existence of a set $H$, with $|H|=\kappa$, that is free for $f$. Given distinct $\alpha,\beta\in H$, the condition $\beta\notin f(\alpha)\cap H$ translates to $\beta\notin N_\alpha$, \emph{i.e.}, $\langle x_\alpha^*,e_\beta \rangle=0$. Consequently, the system $\{e_\alpha;x_\alpha^*\} _{\alpha\in H}$ is biorthogonal, and we are done.
\end{proof}

Having at our disposal those general results concerning the existence of Auerbach systems, we now pass to a basic proposition showing how to obtain separated families of unit vectors starting from a long Auerbach system.
\begin{proposition}\label{prop: Auerbach gives 1+} Suppose that the Banach space $X$ contains an Auerbach system of cardinality $\mathfrak{c}^+$. Then both $S_X$ and $S_{X^*}$ contain an uncountable symmetrically $(1+)$-separated subset.
\end{proposition}
\begin{proof} Clearly, if $\{e_\gamma;\p_\gamma\}_{\gamma\in\Gamma}$ is an Auerbach system in $X$, then we can consider $\{\p_\gamma;e_\gamma\}_{\gamma\in\Gamma}$ as an Auerbach system in $X^*$; consequently, it suffices to prove the result for $X$. Let therefore $\{e_\alpha;\p_\alpha\}_{\alpha<\mathfrak{c}^+}$ be an Auerbach system in $X$ and consider the following colouring $c\colon [\mathfrak{c}^+]^2\to\{(>;>),\,(>;\leq),\,(\leq)\}$: $$\{\alpha,\beta\}\mapsto\begin{cases} 
(>;>) & \|e_\alpha -e_\beta\|>1,\,\|e_\alpha +e_\beta\|>1 \\
(>;\leq) & \|e_\alpha -e_\beta\|>1,\,\|e_\alpha +e_\beta\|\leq1 \\ 
(\leq) & \|e_\alpha -e_\beta\|\leq1.\end{cases}$$
The Erd\H{o}s--Rado theorem assures us of the validity of $\mathfrak{c}^+\to(\omega_1)_3^2$, whence the colouring $c$ admits a monochromatic set $\Lambda\subseteq\mathfrak{c}^+$ with cardinality $\omega_1$. Let us, for notational simplicity, well order the set $\Lambda$ in an $\omega_1$-sequence, thereby obtaining an $\omega_1$-sequence $(e_\alpha)_{\alpha <\omega_1}$ of unit vectors (with the corresponding norm-one biorthogonal functionals $(\p_\alpha)_{\alpha <\omega_1}$) with the property that either $\|e_\alpha-e_\beta\|>1$, $\|e_\alpha+e_\beta\|>1$ for distinct $\alpha,\beta<\omega_1$, or $\|e_\alpha-e_\beta\|>1$, $\|e_\alpha+e_\beta\|\leq1$ for every such $\alpha,\beta$, or finally $\|e_\alpha-e_\beta\|\leq1$ for such $\alpha,\beta$. In the first case, the family $\{e_\alpha\}_{\alpha\in\Lambda}$ is obviously symmetrically $(1+)$-separated, and we are done.\smallskip

In the second case, where $\|e_\alpha-e_\beta\|>1$ and $\|e_\alpha+e_\beta\|\leq1$ for distinct $\alpha,\beta<\omega_1$, we consider, for $1\leq\alpha<\omega_1$ the vector
$$\tilde{e}_\alpha:=e_0 + e_\alpha.$$
Our assumption evidently implies that these are, indeed, unit vectors. Moreover, for distinct $1\leq\alpha,\beta<\omega_1$ we have
$$\|\tilde{e}_\alpha-\tilde{e}_\beta\|=\|e_\alpha-e_\beta\|>1$$
$$\|\tilde{e}_\alpha+\tilde{e}_\beta\|=\|2e_0+e_\alpha+e_\beta\|\geq \langle\p_0,2e_0+e_\alpha+e_\beta\rangle=2,$$
whence the vectors $\{\tilde{e}_\alpha\}_{1\leq\alpha <\omega_1}$ are symmetrically $(1+)$-separated.\smallskip

Finally, let us consider the case where  $\|e_\alpha-e_\beta\|\leq1$ whenever $\alpha,\beta<\omega_1$ (regardless of the value of $\|e_\alpha+e_\beta\|$). We may then modify the vectors $e_\alpha$ as follows: for $1\leq\alpha<\omega_1$, we set
$$\tilde{e}_\alpha:=e_\alpha-\sum_{\gamma<\alpha}c_\alpha^\gamma e_\gamma,$$
where $\{c_\alpha^\gamma\}_{\gamma<\alpha}$ is a collection of positive real numbers such that, for $1\leq\alpha<\omega_1$, $c_\alpha^0\geq3/4$ and $\sum_{\gamma<\alpha}c_\alpha^\gamma=1$ (such collections do exist since the set $\{\gamma\colon\gamma <\alpha\} $ is countable).\smallskip

Let us first observe that this modification does not change the norm of the vectors: indeed, on the one hand 
$$\|\tilde{e}_\alpha\|=\left\Vert\sum_{\gamma<\alpha}c_\alpha^\gamma e_\alpha-\sum_{\gamma<\alpha}c_\alpha^\gamma e_\gamma\right\Vert\leq\sum_{\gamma<\alpha} c_\alpha^\gamma\|e_\alpha-e_\gamma\|\leq1;$$
on the other hand, $\langle\p_\alpha,\tilde{e}_\alpha\rangle=\langle\p_\alpha, e_\alpha\rangle=1$. Consequently, $\tilde{e}_\alpha\in S_X$.

Finally, the family $\{\tilde{e}_\alpha\}_{1\leq\alpha <\omega_1}$ is symmetrically $(1+)$-separated. Indeed, for any choice $1\leq\alpha<\beta<\omega_1$, we have
$$\|\tilde{e}_\alpha-\tilde{e}_\beta\|\geq\left|\langle\p_\alpha,\tilde{e}_\alpha-\tilde{e}_\beta \rangle\right| =\langle\p_\alpha,e_\alpha\rangle-\left\langle\p_\alpha,e_\beta-\sum_{\gamma<\beta} c^\gamma_\beta e_\gamma\right\rangle=1+c^\alpha_\beta>1,$$
$$\|\tilde{e}_\alpha+\tilde{e}_\beta\|\geq|\langle\p_0,\tilde{e}_\alpha+\tilde{e}_\beta \rangle|=|c_\alpha^0+c_\beta^0|\geq3/2.$$
\end{proof}

\begin{remark} It is perhaps worth observing that the above proof also yields us an alternative, shorter proof of the symmetric version of Kottman's theorem, \cite[Theorem A]{HKR-sym-sep} (although the argument there is self-contained and it contains no combinatorial features). Indeed, it is sufficient to replace the use of the Erd\H{o}s--Rado theorem with Ramsey's theorem, in the form $\omega\to(\omega)_3^2$, and remember that every infinite-dimensional Banach space contains an infinite Auerbach system, \cite{Day}. The argument then proceeds identically, with only obvious notational modifications.
\end{remark}

If we combine the above proposition with the previous statements concerning the existence of Auerbach systems, we immediately arrive at the following results.
\begin{corollary}\label{(1+) in large space} Let $X$ be a Banach space with $w^*\text{-}{\rm dens}\, X^*>\exp_2\mathfrak{c}$. Then both $X$ and $X^*$ contain uncountable symmetrically $(1+)$-separated families of unit vectors.

In particular, the unit sphere of every Banach space $X$ with ${\rm dens}\, X>\exp_3\mathfrak{c}$ contains an uncountable symmetrically $(1+)$-separated subset. 
\end{corollary}

\begin{corollary}\label{(1+) in WLD} Let $X$ be a WLD Banach space with ${\rm dens}\, X>\mathfrak{c}$. Then the unit spheres of $X$ and of $X^*$ contain uncountable symmetrically $(1+)$-separated subsets.
\end{corollary}

In conclusion to this section, we shall present a few, much simpler, renorming results, which also depend on the existence of Auerbach systems (or, more generally, biorthogonal systems). Those results are the non-separable counterparts to \cite[Section 5.1]{HKR-sym-sep}, with essentially the same proofs. Let us also mention \cite[Theorem 3]{MV equilateral}, where similar renorming techniques are shown to produce infinite equilateral sets.

Recall that a biorthogonal system $\{x_i;f_i\}_{i\in I}$ in $X$ is said to be \emph{bounded} if
$$\sup_{i\in I}\|x_i\|\cdot\|f_i\|<\infty.$$
Of course, up to a scaling, we can always assume that the system is \emph{normalised}, \emph{i.e.}, such that $\|x_i\|=1$ ($i\in I$).

\begin{proposition}Let $(X,\n)$ be a Banach space that contains a bounded biorthogonal system $\{x_i;f_i\}_{i\in I}$. Then there exists an equivalent norm $\nn\cdot$ on $X$ such that $S_{(X,\nn\cdot)}$ contains a symmetrically $2$-separated subset with cardinality $\abs{I}$. In particular, $\{x_i\}_{i\in I}$ is such a set.
\end{proposition}
The proof is the same as in \cite[Proposition 5.1]{HKR-sym-sep}, and it is therefore omitted. Let us just mention that the norm $\nn\cdot$ can be explicitly defined as
$$\nn{x}=\max\left\{\sup_{i\neq k\in I}\left(\big|\langle f_i,x\rangle\big|+\big|\langle f_k,x\rangle\big|\right),\|x\|\right\} \quad (x\in X)$$
and that the above renorming was already present in the proof of \cite[Theorem 7]{Kottman}. \smallskip 

Note that, if $\{x_i;f_i\}_{i\in I}$ is a biorthogonal system and $\abs{I}$ has uncountable cofinality, we can pass to a subsystem with the same cardinality and which is bounded. In fact the sets $I_n:=\{i\in I\colon\|x_i\|\cdot\|f_i\|\leq n\}$ satisfy $\cup_{n=1}^\infty I_n = I$; hence, for some $n\in\N$, we have $\abs{I_n}=\abs{I}$ and, of course, $\{x_i;f_i\}_{i\in I_n}$ is a bounded biorthogonal system.\smallskip

We can now combine those simple observations with deep results concerning the existence of uncountable biorthogonal systems in non-separable Banach spaces. In his fundamental work \cite{Todorcevic}, Todor\u{c}evi\'c has proved the consistency with ZFC of the claim that every non-separable Banach space admits an uncountable biorthogonal system, proving in particular the result under Martin's Maximum (MM) (\cite[Corollary 7]{Todorcevic}). 

Let us mention--even if not strictly needed--that under suitable additional set theoretic assumptions there exist non-separable Banach spaces with no uncountable biorthogonal systems. The first such example, under (CH), is due to Kunen (unpublished) and appeared later in the survey \cite{Negr84}; other published results, under $\clubsuit$ and $\diamondsuit$ respectively, are \cite{Ost76,Shel85}. We also refer to \cite[Section 4.4]{HMVZ} for a modification, suggested by Todor\u{c}evi\'c, of the argument in \cite{Ost76}. Let us also refer to \cite{FaGo88,FiGo89,GoTa82,Laza81,Steg75,Todorcevic}, or \cite[Section 4.3]{HMVZ}, for some absolute (ZFC) results.

\begin{corollary}\label{uncountable sym-2-sep under renorming}It is consistent with ZFC that every non-separable Banach space $X$ admits an equivalent norm $\nn\cdot$ such that $S_{(X,\nn\cdot)}$ contains an uncountable symmetrically $2$-separated subset.
\end{corollary}

In the particular case that the biorthogonal system is an Auerbach system, we can specialise the above renorming and obtain an approximation of the original norm.
\begin{proposition}\label{Auerbach and renorming} Assume that a Banach space $(X,\n)$ contains an Auerbach system $\{x_i,f_i\}_{i\in I}$. Then, for every $\e>0$, $X$ admits an equivalent norm $\nn\cdot$ such that $\n\leq\nn\cdot\leq(1+\e)\n$ and (for some $\delta>0$) $S_{(X,\nn\cdot)}$ contains a symmetrically $(1+\delta)$-separated subset with cardinality $\abs{I}$.
\end{proposition}

\subsection{$c_0(\Gamma)$ spaces}
In this chapter we shall investigate the existence of separated families of unit vectors in spaces of the form $c_0(\Gamma)$. Such spaces are natural candidates for investigation since they constitute the archetypal examples of spaces which fail to contain uncountable $(1+\e)$-separated families of unit vectors. This was already pointed out by Elton and Odell (\cite[Remark (2)]{E-O}); the very simple proof is also recorded in \cite[Proposition 2.1]{KaKo}. On the other hand, it was noted in \cite[Remark (2), p.~558]{GlaMer} that, for uncountable $\Gamma$, the unit sphere of $c_0(\Gamma)$ contains an uncountable $(1+)$-separated subset; it even contains an uncountable symmetrically $(1+)$-separated subset. The very quick argument is also included here.

\begin{example} For every uncountable set $\Gamma$, $c_0(\Gamma)$ contains an uncountable
symmetrically $(1+)$-separated family of unit vectors.\smallskip

Of course, it suffices to prove the claim for $\Gamma=\omega_1$. For $1\leq\alpha<\omega_1$, we choose a unit vector $x_\alpha\in c_0(\omega_1)$ such that 
$$x_\alpha(\lambda)=\begin{cases} 1 & \lambda=0,\alpha\\ <0 & 1\leq\lambda<\alpha\\ 
0 & \alpha<\lambda<\omega_1;\end{cases}$$
such a choice is indeed possible since, for $\alpha<\omega_1$, the set $\{\lambda\colon 1\leq\lambda<\alpha\}$ is at most countable. It is obvious that the family $(x_\alpha)_{1\leq\alpha<\omega_1}$ is symmetrically $(1+)$-separated, since, for $1\leq\alpha<\beta<\omega_1$, we have $\|x_\alpha-x_\beta\|\geq|x_\alpha(\alpha)-x_\beta(\alpha)|=1-x_\beta(\alpha)>1$ and $\|x_\alpha+x_\beta\|\geq|x_\alpha(0)+x_\beta(0)|=2$.
\end{example}

Our next result shows that the above obvious construction can not be improved, as every $(1+)$-separated family of unit vectors in a $c_0(\Gamma)$ space has cardinality at most $\omega_1$. This result improves an observation due to P.~Koszmider (see \cite[Proposition 4.13]{KaKo}). Just as the assertion about $(1+\e)$-separation, the proof exploits the $\Delta$-system lemma, whose statement was recalled in Section \ref{Review combinatory}.

\begin{theorem}\label{c0omega1} Let $A\subseteq S_{c_{0}(\Gamma)}$ be a $(1+)$-separated set. Then $|A|\leq\omega_1$.
\end{theorem}
\begin{proof} The assertion is trivially true for $|\Gamma|\leq\omega_1$; on the other hand, if $|\Gamma|>\omega_2$ and the unit sphere of $c_0(\Gamma)$ contains a $(1+)$-separated family of cardinality $\omega_2$, then the union of the supports of those vectors is a set with cardinality at most $\omega_2$; therefore, we would find a $(1+)$-separated family with cardinality $\omega_{2}$ in the unit sphere of $c_{0}(\omega_{2})$. Consequently, we may without loss of generality restrict our attention to the case $\Gamma=\omega_{2}$.\smallskip

Assume, in search for a contradiction, that the unit sphere of $c_0(\omega_2)$ contains a subset $\{x_\alpha\} _{\alpha<\omega_2}$ such that $\|x_\alpha-x_\beta\|>1$ for every choice of distinct $\alpha,\beta<\omega_2$. For $\alpha<\omega_2$, consider the finite sets $N_\alpha:=\{|x_\alpha|\geq 1/2\}$, where $\{|x_\alpha|\geq 1/2\}$ is a shorthand for the set $\{\gamma<\omega_2\colon |x_\alpha(\gamma)|\geq1/2\}$. The $\Delta$-system lemma allows us to assume (up to passing to a~subset that still has cardinality $\omega_2$) that there is a finite subset $\Delta$ of $\omega_2$ such that $N_\alpha\cap N_\beta=\Delta$ whenever $\alpha\neq\beta$; moreover, using again the regularity of $\omega_2$, we can also assume that all the $N_\alpha$'s have the same (finite) cardinality. Since $\Delta$ is a finite set, the unit ball of $\ell_\infty(\Delta)\subseteq c_0(\omega_2)$ is compact and it can be covered by finitely many balls of radius $1/2$; also, $x_{\alpha|\Delta}\in B_{\ell_\infty(\Delta)}$ for every $\alpha$. These two facts imply that there is a subset $\Omega$ of $\omega_2$, still with cardinality $\omega_2$, such that all $x_{\alpha|\Delta}$'s ($\alpha\in\Omega$) lie in the same ball; in other words, up to passing to a~further subset, we can assume that, for every $\alpha,\beta<\omega_2$, we have $\|x_{\alpha|\Delta}-x_{\beta|\Delta}\|\leq1$.\smallskip

Let us summarise what we have obtained so far. If, by contradiction, the conclusion of the theorem is false, then there is a $(1+)$-separated subset $\{x_\alpha\}_{\alpha<\omega_2}$ of $S_{c_0(\omega_2)}$ such that the following hold true:
\begin{romanenumerate}
\item there is a finite set $\Delta\subseteq\omega_2$ with $N_\alpha\cap N_\beta=\Delta$ for distinct $\alpha,\beta<\omega_2$;
\item the sets $N_\alpha$ have the same finite cardinality;
\item $\|x_{\alpha|\Delta}-x_{\beta|\Delta}\|\leq1$ for every $\alpha,\beta<\omega_2$.
\end{romanenumerate}\smallskip

We will show that these properties lead to a contradiction. According to (i), we may write $N_\alpha=\Delta\cup\tilde{N}_\alpha$, where $\tilde{N}_\alpha\subseteq \omega_2\setminus\Delta$ and the sets $\tilde{N}_\alpha$ are mutually disjoint. Moreover, by (ii), they also have the same finite cardinality, say $k$; let us then write $\tilde{N}_\alpha=\{\lambda_1^\alpha,\dots,\lambda_k^\alpha\}$ with $\lambda_1^\alpha<\lambda_2^\alpha<\dots<\lambda_k^\alpha<\omega_2$. The disjointness of the $\tilde{N}_\alpha$'s forces in particular $\lambda_1^\alpha\neq\lambda_1^\beta$ for $\alpha\neq\beta$ and this in turn implies
\begin{equation}\label{sup=omega_2}
\sup_{\alpha<\omega_2}\lambda_1^\alpha=\omega_2.
\end{equation}

We then consider the set $\cup_{\alpha<\omega_1}{\rm supp}x_\alpha$; such a set has cardinality at most $\omega_1$, so its supremum is necessarily strictly smaller than $\omega_2$. Therefore, combining this information and (\ref{sup=omega_2}), we infer that there exists an ordinal $\beta<\omega_2$ such that $$\sup\left(\bigcup_{\alpha<\omega_1}{\rm supp}x_\alpha\right)<\lambda_1^\beta;$$ in particular, this implies that $\tilde{N}_\beta\cap {\rm supp}x_\alpha=\varnothing$ for every $\alpha<\omega_1$. Moreover, the set $\{0< |x_\beta|<1/2\}$ is of course countable, whence it can intersect at most countably many of the disjoint sets $\{\tilde{N}_\alpha\}_{\alpha<\omega_1}$. Consequently, we can find an ordinal $\alpha<\omega_1$ such that $\{0<|x_\beta|<1/2\}\cap\tilde{N}_\alpha=\varnothing$ too. In order to understand where the supports of those $x_\alpha$ and $x_\beta$ could possibly intersect, we note that for every $\gamma$ we have the disjoint union $${\rm supp}\, x_\gamma=\Delta\cup\left(\tilde{N}_\gamma\cup\{0<|x_\gamma|<1/2\}\right).$$ Therefore, using our previous choices of $\beta$ and $\alpha$, we obtain:
$${\rm supp}x_\alpha\cap{\rm supp}x_{\beta}=\Delta\cup\left(
\left(\tilde{N}_\alpha\cup\{0<|x_\alpha|<1/2\}\right) \cap
\left(\tilde{N}_\beta\cup\{0<|x_{\beta}|<1/2\}\right) \right)$$
$$=\Delta\cup\left(\{0<|x_\alpha|<1/2\}\cap\{0<|x_\beta|<1/2\}\right).$$
Finally, for every $\gamma\in\{0<|x_\alpha|<1/2\}\cap\{0<|x_\beta|<1/2\}$ it is obvious that we have $|x_\alpha(\gamma)-x_\beta(\gamma)|\leq1$; consequently, the condition $\|x_\alpha-x_\beta\|>1$ can only be witnessed by coordinates from $\Delta$, \emph{i.e.}, $\| x_{\alpha|\Delta}-x_{\beta|\Delta}\|>1$; however, this readily contradicts (iii).
\end{proof}

\begin{remark} Let us notice in passing that if we consider the spaces $c_{00}(\Gamma)$, then every $(1+)$-separated family of unit vectors is actually at most countable; this is also pointed out in \cite[Remark (2), p. 558]{GlaMer} and it immediately follows from the $\Delta$-system lemma exactly as in \cite[Remark (2)]{E-O}.
\end{remark}

The results mentioned so far in this section draw a complete picture about separated families of unit vectors in $c_0(\Gamma)$ spaces; we therefore conclude this section presenting some results concerning renormings of those spaces.\smallskip

Proposition \ref{Auerbach and renorming} in the previous section applies in particular to the canonical basis of $c_0(\Gamma)$; consequently, the canonical norm $\n_\infty$ on $c_0(\Gamma)$ can be approximated by norms whose unit spheres contain (for some $\e>0$) symmetrically $(1+\e)$-separated subsets of cardinality $|\Gamma|$. If we combine this result with James' non distortion theorem, we obtain the approximation of every equivalent norm on $c_0(\Gamma)$. Before we proceed, a historical remark about the non-separable counterparts of James' theorems is in order.\smallskip

\begin{remark}It was communicated to us by W.~B.~Johnson that the non-separable analogue of both James' non-distortion theorems were known to the experts immediately after the paper of James \cite{james-distortion} had been published. At the beginning of this century, A.~S.~Granero being unaware of this situation, circulated a note containing the proofs of these theorems (\cite{granero}). It seems that the first published proof of non-separable versions of James's non-distortion theorems may be found in \cite[Theorem 3]{hajeknovotny}.\end{remark}

\begin{proposition} Every equivalent norm on $c_0(\Gamma)$ can be approximated by norms whose unit spheres contain (for some $\delta>0$) symmetrically $(1+\delta)$-separated subsets of cardinality $|\Gamma|$. 
\end{proposition}
\begin{proof}
Let $\n$ be any equivalent norm on $c_0(\Gamma)$ and fix $\e>0$. By the non-separable version of James' non-distortion theorem, there exists a subspace $Y$ of $c_0(\Gamma)$, with ${\rm dens}(Y)=|\Gamma|$, such that $Y$ is $\e$-isometric to $(c_0(\Gamma), \n_\infty)$. Let $T\colon(Y,\n)\rightarrow(c_0(\Gamma), \n_\infty)$ be a linear isomorphism witnessing this fact; in particular, we may assume that $\|T\|\leq1+\e$ and $\|T^{-1}\|\leq1$. According to Proposition \ref{Auerbach and renorming}, we can choose a norm $\nu$ on $(c_0(\Gamma), \n_\infty)$ with $\n_\infty\leq\nu\leq(1+\e)\n_\infty$ and such that $S_{(c_0(\Gamma), \nu)}$ contains (for some $\delta>0$) a symmetrically $(1+\delta)$-separated family of cardinality $|\Gamma|$.

Consider a new norm on $Y$ given by $y\mapsto\nu(Ty)$ ($y\in Y$); since
$$\|y\|\leq\|Ty\|_\infty\leq\nu(Ty)\leq(1+\e)\|Ty\|_\infty\leq(1+\e)^2\|y\|,$$
it is well known that we can extend the norm $\nu\circ T$ to a norm $\nn\cdot$ defined on $(c_0(\Gamma), \n)$ and still satisfying $\n\leq\nn\cdot\leq(1+\e)^2\n$. Finally, $T$ is an isometry from $(Y,\nn\cdot)$ onto $(c_0(\Gamma), \nu)$, whence the unit sphere of $(c_0(\Gamma), \nn\cdot)$ contains a symmetrically $(1+\delta)$-separated family of cardinality $|\Gamma|$.
\end{proof}
\begin{remark} Let us note in passing that, by a very similar argument, the unit sphere of every equivalent norm on $\ell_1(\Gamma)$ contains (for some $\e>0$) a symmetrically $(1+\e)$-separated family of cardinality $|\Gamma|$.
\end{remark}

In the last result for this section we provide a sufficient condition for a renorming of $c_0(\Gamma)$ to contain an uncountable $(1+)$-separated family of unit vectors; the argument elaborates over the proof of Proposition \ref{prop: Auerbach gives 1+}. Further sufficient conditions will follow from the results in Section \ref{Section: geometry}. \smallskip

Let $\Gamma$ be any set and $\n$ be a norm on $c_0(\Gamma)$ (not necessarily equivalent to the canonical $\n_\infty$ norm of $c_0(\Gamma)$). We say that $\n$ is a \emph{lattice norm} if $\|x\|\leq\|y\|$ for every pair of vectors $x,y\in c_0(\Gamma)$ such that $|x(\gamma)|\leq|y(\gamma)|$ for each $\gamma\in\Gamma$. In other words, $\n$ is a lattice norm if $(c_0(\Gamma),\n)$ is a normed lattice when endowed with the canonical coordinate-wise partial ordering. Accordingly, in what follows we shall denote by $|x|$ the element defined by $|x|(\gamma):=|x(\gamma)|$ ($\gamma\in\Gamma$).

\begin{proposition} Let $\Gamma$ be an uncountable set and $\n$ be a (not necessarily equivalent) lattice norm on $c_0(\Gamma)$. Then the unit sphere of $(c_0(\Gamma),\n)$ contains an uncountable $(1+)$-separated subset.
\end{proposition}
\begin{proof} It is sufficient to prove the result for $\Gamma=\omega_1$. Let us denote by $e_\alpha$ ($\alpha<\omega_1$) the $\alpha$-th element of the canonical basis, \emph{i.e.}, $e_\alpha(\gamma):=\delta_{\alpha,\gamma}$; we also denote by $\tilde{e}_\alpha$ the unit vector $\tilde{e}_\alpha:=e_\alpha/\|e_\alpha\|$. We may now choose, for every $\beta<\omega_1$, real numbers $(c^\alpha_\beta)_{\alpha<\beta}$ subject to the following three conditions:
\begin{romanenumerate}
\item $c^\alpha_\beta\geq0$ for each $\alpha<\beta$;
\item $c^\alpha_\beta>0$ if and only if $\|\tilde{e}_\alpha-\tilde{e}_\beta\|\leq1$;
\item if $c^\alpha_\beta>0$ for some $\alpha<\beta$, then $\sum_{\gamma<\beta}c^\gamma_\beta=1$.
\end{romanenumerate}
Observe that condition (iii) could be equivalently stated as the requirement $\sum_{\alpha<\beta}c^\alpha_\beta$ to equal either $0$ or $1$; note, further, that (ii) implies $c^\alpha_\beta\cdot\|\tilde{e}_\alpha-\tilde{e}_\beta\|\leq c^\alpha_\beta$ for each $\alpha<\beta$.\smallskip

We are now in position to define vectors $f_\beta\in c_0(\omega_1)$ ($\beta<\omega_1$) as follows:
$$f_\beta:=\tilde{e}_\beta-\sum_{\alpha<\beta}c^\alpha_\beta\,\tilde{e}_\alpha.$$
We readily verify that $\|f_\beta\|=1$. In fact, if $\sum_{\alpha<\beta}c^\alpha_\beta=0$, then $f_\beta=\tilde{e}_\beta$ and there is nothing to prove. In the other case, \emph{i.e.}, $\sum_{\alpha<\beta}c^\alpha_\beta=1$, we have, according to (ii),
$$\|f_\beta\|=\left\|\sum_{\alpha<\beta}c^\alpha_\beta\,\tilde{e}_\beta-\sum_{\alpha<\beta} c^\alpha_\beta\,\tilde{e}_\alpha\right\|\leq\sum_{\alpha<\beta} c^\alpha_\beta\|\tilde{e}_\beta-\tilde{e}_\alpha\|\leq\sum_{\alpha<\beta}c^\alpha_\beta =1.$$
On the other hand, $|f_\beta|\geq|\tilde{e}_\beta|$, whence $\|f_\beta\|=1$ follows from the lattice property.\smallskip

To conclude, we prove that the vectors $(f_\alpha)_{\alpha<\omega_1}$ are $(1+)$-separated. Given $\alpha<\beta<\omega_1$, we distinguish two cases. If $c^\alpha_\beta=0$, then by our previous choice we have $\|\tilde{e}_\alpha-\tilde{e}_\beta\|>1$. Moreover, $|f_\alpha-f_\beta|\geq|\tilde{e}_\alpha-\tilde{e}_\beta|$ and the lattice property imply $\|f_\alpha-f_\beta\|\geq\|\tilde{e}_\alpha-\tilde{e}_\beta\|>1$. On the other hand, if $c^\alpha_\beta>0$, we note that $|f_\alpha-f_\beta|\geq|(1+c^\alpha_\beta)\tilde{e} _\alpha|$; consequently, exploiting once more the lattice property, we conclude that $\|f_\alpha-f_\beta\|\geq\|(1+c^\alpha_\beta)\tilde{e}_\alpha\|=1+c^\alpha_\beta>1$.
\end{proof}
\begin{remark} Concerning lattice norms, we would like to point out here the validity of the analogous result for lattice norms on the space $C([0,\omega_1])$. Inspection of the proof of \cite[Proposition 2]{FHZ} shows that if $\n$ is any (not necessarily equivalent) lattice norm on $C([0,\omega_1])$, then $(C([0,\omega_1]),\n)$ contains an isometric copy of $(c_0(\omega_1),\n_\infty)$. Consequently, the unit sphere of $(C([0,\omega_1]),\n)$ contains an uncountable $(1+)$-separated subset.
\end{remark}

\section{A geometric approach involving exposed points}\label{Section: geometry}
In the present chapter we shall be concerned with the proof of Theorem B; both clauses (i) and (ii) will depend on rather powerful results concerning the notion of an exposed point (which were recorded in Section \ref{review exposed}), whence the title of this section. We will also present an abstract and a more general result, whose proof follows an analogous pattern, and show some its consequences. The material is naturally divided into two parts, the former concerning $(1+)$-separation and the latter investigating $(1+\e)$-separation.
Finally, in the last section, we shall prove claim (iii) of theorem B.

\subsection{$(1+)$-separation}\label{Section: geometry (1+)}
An important result for the structure of quasi-reflexive Banach spaces is the fact that every quasi-reflexive Banach space $X$ contains a reflexive subspace $Y$ with the same density character as $X$. For non-separable $X$ this was proved in \cite[Theorem 4.6]{CiHo}, while the assertion in the separable case follows from some results by Johnson and Rosenthal (\emph{cf.} \cite[Corollary IV.1]{JoRo}). Consequently, the next result implies assertion (i) in Theorem B. 

\begin{theorem}\label{(1+) in reflexive} Let $X$ be an infinite-dimensional, reflexive Banach space. Then the unit sphere of $X$ contains a symmetrically $(1+)$-separated subset of cardinality ${\rm dens}\, X$.
\end{theorem}

\begin{proof}Let $X$ be an infinite-dimensional, reflexive Banach space and set $\lambda={\rm dens}\, X$. According to the result by Lindenstrauss and Troyanski quoted above, the weakly compact set $B_X$ contains an exposed point $x_0$; we can then choose a functional $\p_0\in X^*$ that exposes $x_0$ (and we let $X_0:=X$ for notational consistency). Needless to say, the subspace $X_1:=\ker\p_0\subseteq X_0$ is a reflexive Banach space, whence $B_{X_1}$ is a weakly compact subset of $X$. Consequently, we can choose an exposed point $x_1\in B_{X_1}$ and a functional $\p_1\in X^*$ that exposes $x_1$. We now proceed by transfinite induction. Assume, for some $\beta<\lambda$, to have already found closed subspaces $(X_\alpha)_{\alpha<\beta}$ of $X$, unit vectors $(x_\alpha)_{\alpha<\beta}$ and functionals $(\p_\alpha)_{\alpha<\beta}$ such that for each $\alpha<\beta$:
\begin{romanenumerate}
\item $X_\alpha=\bigcap_{\gamma<\alpha}\ker\p_\gamma$;
\item $x_\alpha\in X_\alpha$ is an exposed point for $B_{X_\alpha}$ in $X$ and the functional $\p_\alpha\in X^*$ exposes $x_\alpha$. 
\end{romanenumerate}

Consider the closed subspace $X_\beta:=\bigcap_{\alpha<\beta}\ker\p_\alpha$. The reflexivity of $X$ implies that 
$w^*\text{-}{\rm dens}\, X^*={\rm dens}\, X=\lambda>\abs{\beta}$, whence the linear span of the family 
$\{\p_\alpha\}_{\alpha<\beta}$ can not be $w^*$-dense in $X^*$. Consequently, 
$\{\p_\alpha\}_{\alpha<\beta}$ does not separate points in $X$ and $X_\beta$ does not reduce to the zero vector. Moreover, $B_{X_\beta}$ is a weakly compact subset of $X$, hence it admits an exposed point $x_\beta$, exposed by $\p_\beta\in X^*$. The fact that $X_\beta\neq\{0\}$ ensures us that $x_\beta$ is a unit vector. This completes the inductive step and shows the existence of families of closed subspaces $(X_\alpha)_{\alpha<\lambda}$ of $X$, unit vectors $(x_\alpha)_{\alpha<\lambda}$ and functionals $(\p_\alpha)_{\alpha<\lambda}$ with the two properties above.

\smallskip
In conclusion, we show that the family $(x_\alpha)_{\alpha<\lambda}$ is symmetrically $(1+)$-separated: if $\alpha<\beta<\lambda$, by construction $x_\beta\in X_\beta\subseteq\ker\p_\alpha$, so $\langle\p_\alpha,x_\alpha\pm x_\beta\rangle=\langle\p_\alpha,x_\alpha\rangle$. Since $\p_\alpha$ exposes $x_\alpha$, this last equality implies that $x_\alpha \pm x_\beta \notin B_{X_\alpha}$. But of course $x_\alpha \pm x_\beta \in X_\alpha$ and consequently $\| x_\alpha \pm x_\beta \| >1$, thus concluding the proof.
\end{proof}

It is perhaps clear that the above reasoning can be adapted to more general spaces, notably spaces with the RNP. However, instead of giving various related results with similar proofs, we prefer to present an abstract version of the present reasoning, which subsumes many concrete examples. We then present concrete consequences of this general theorem.

\begin{definition} An infinite-dimensional Banach space $(X,\n)$ is said to admit the \emph{point with small flatness property} (shortly, \emph{PSF}) if there exist $x\in S_X$ and a closed subspace $Y$ of $X$, with $\dim(X/Y)<\infty$, such that $\|x+y\|>1$ for every unit vector $y\in S_Y$. In symbols,
$$X \text{ has PSF whenever:}\quad \begin{cases}\exists x\in S_X, \exists Y\subseteq X \text{ closed subspace, with}\dim(X/Y)<\infty:\\ \forall y\in S_Y \, \|x+y\|>1.\end{cases}$$
\end{definition}

This notion is inspired from condition (\Square) in the proof of \cite[Theorem A]{HKR-sym-sep}. We are going to see in a moment various examples of spaces with property (PSF), therefore it is perhaps worth including here an example of a Banach space failing this property.
\begin{example} The space $X=c_0(\omega_1)$ is an easy example of a space failing (PSF). Indeed, fix arbitrarily $x\in S_X$ and a finite-codimensional closed subspace $Y$ of $X$. For $\alpha<\omega_1$, we denote by $X_{>\alpha}$ the closed subspace of $X$ consisting of all vectors whose support is contained in $[\alpha+1,\omega_1)$. Of course, there exists an ordinal $\alpha<\omega_1$ such that ${\rm supp}\,x\leq\alpha$ and, moreover, $Y\cap X_{>\alpha}$ is a finite-codimensional subspace of $X_{>\alpha}$. In particular, if we choose $y\in Y\cap X_{>\alpha}$ with $\|y\|=1$, $x$ and $y$ are disjointly supported, whence $\|x+y\|=1$.

Let us also mention that there exists examples of separable Banach spaces that fail (PSF). In particular, it is possible to show that the space $c_0$ does not satisfy property (PSF). The proof of this claim is similar to the one above, together with some linear algebra calculations; therefore, we prefer to omit it. 
\end{example}



Our general result now reads as follows.
\begin{theorem}\label{general th for (1+)} Let $X$ be an infinite-dimensional Banach space such that every infinite-dimensional subspace $\tilde{X}$ of $X$ has (PSF). Then $S_X$ contains a symmetrically $(1+)$-separated subset with cardinality $w^*\text{-}{\rm dens}\, X^*$.
\end{theorem}

\begin{proof} Let $\lambda:=w^*\text{-}{\rm dens}\, X^*$. If $\lambda=\omega$, then the result is contained in \cite[Theorem A]{HKR-sym-sep}, so we assume $\lambda>\omega$. We shall construct by transfinite induction a family of unit vectors $(x_\alpha)_{\alpha<\lambda}\subseteq S_X$ and a decreasing family of closed subspaces $(H_\alpha)_{\alpha<\lambda}$ of $X$ such that, for every $\alpha<\lambda$:
\begin{romanenumerate}
\item $\|x_\alpha+y\|>1$ for every $y\in S_{H_\alpha}$;
\item $x_\beta\in H_\alpha$ for $\alpha<\beta<\lambda$;
\item $\dim((\cap_{\beta<\alpha}H_\beta)/H_\alpha)<\infty$ for every $\alpha<\lambda$.
\end{romanenumerate}
Condition (iii) is a technical condition that we need for the inductive procedure to continue until $\lambda$; it implies in particular that $\dim(X/H_0)<\infty$ \footnote{Here, as it is customary, we understand the empty intersection $\cap_{\beta<0}H_\beta$ to be equal to the whole $X$.} and $\dim(H_\alpha/H_{\alpha+1})<\infty$ for each $\alpha<\lambda$.
Once such a family is constructed, for $\alpha<\beta<\lambda$ we have $\pm x_\beta\in H_\alpha$, by (ii). From (i) we conclude that $\|x_\alpha\pm x_\beta\|>1$, whence the family $(x_\alpha)_{\alpha<\lambda}$ is symmetrically $(1+)$-separated.\smallskip

Before entering the construction of those families, we note the following: if condition (iii) is satisfied for every $\alpha<\gamma$, then $\cap_{\alpha<\gamma}H_\alpha$ is the intersection of at most $\max\{|\gamma|,\omega\}$ kernels of functionals from $X$, \emph{i.e.}, $w^*\text{-}{\rm dens}((X/\cap_{\alpha<\gamma}H_\alpha)^*)\leq\max\{|\gamma|,\omega\}$.\smallskip

The proof of this is also based on a simple transfinite induction argument: assuming the statement to be true for every $\alpha<\delta$ (where $\delta<\gamma$), if $\delta=\delta'+1$ is a successor ordinal, then of course $\cap_{\alpha<\delta}H_\alpha=H_{\delta'}$. Consequently, the facts that $\cap_{\alpha<\delta'}H_\alpha$ is the intersection of at most $\max\{|\delta'|,\omega\}$ kernels of functionals and that $\dim((\cap_{\alpha<\delta'}H_\alpha)/H_{\delta'})<\infty$ imply the desired assertion for $\cap_{\alpha<\delta}H_\alpha$. If $\delta$ is a limit ordinal, then $\cap_{\alpha<\delta}H_\alpha=\cap_{\beta<\delta}(\cap_{\alpha<\beta}H_\alpha)$ and each $\cap_{\alpha<\beta}H_\alpha$ is the intersection of at most $\max\{|\beta|,\omega\}\leq |\delta|$ kernels of functionals; hence the same is true for $\cap_{\alpha<\delta}H_\alpha$. \smallskip

We now turn to the construction of the families $(x_\alpha)_{\alpha<\lambda}$ and $(H_\alpha)_{\alpha<\lambda}$ with the desired properties. Assume by transfinite induction to have already found such elements for every $\alpha<\gamma$, where $\gamma<\lambda$. From the above observation, we infer that 
$$w^*\text{-}{\rm dens}\,(X/\bigcap_{\alpha<\gamma}H_\alpha)^*\leq\max\{|\gamma|,\omega\}<\lambda$$
and this readily implies that $w^*\text{-}{\rm dens}((\cap_{\alpha<\gamma}H_\alpha)^*)=\lambda$ \footnote{Indeed, it is a standard fact that, for a closed subspace $Y$ of a Banach space $X$, 
$$w^*\text{-}{\rm dens}\,X^*\leq \max\left\{ w^*\text{-}{\rm dens}\,Y^*, w^*\text{-}{\rm dens}\,(X/Y)^* \right\}.$$
Let us stress that equality may fail to hold, as witnessed by the fact that $\ell_2(\mathfrak{c})$ is a quotient of $\ell_1(\mathfrak{c})$, yet $w^*\text{-}{\rm dens}\,\ell_1(\mathfrak{c})^*=\omega$ and $w^*\text{-}{\rm dens}\,\ell_2(\mathfrak{c})^*=\mathfrak{c}$.}. In particular, $\cap_{\alpha<\gamma}H_\alpha$ (is infinite-dimensional, hence it) has property (PSF) and we can thus find a unit vector $x_\gamma\in\cap_{\alpha<\gamma}H_\alpha$ and a finite-codimensional subspace $H_\gamma$ of $\cap_{\alpha<\gamma}H_\alpha$ such that $\|x_\gamma+y\|>1$ for every $y\in S_{H_\gamma}$. (Note that in the case $\gamma=0$ the above argument just reduces to exploiting the (PSF) property of $X$.) It is then clear that properties (i)--(iii) are satisfied with such a choice; therefore, the transfinite induction step is complete and we are done.
\end{proof}

\begin{definition} Let $X$ be a normed space and $F$ be a non-empty subset of $S_X$. We say that $F$ is a \emph{face} of $B_X$ if there exists a functional $\p\in S_{X^*}$ such that $F=B_X\cap\{\langle\p,\cdot\rangle=1\}$.
\end{definition}
Note that if $S_X$ contains a face $F$ with ${\rm diam}(F)<1$, then $X$ has (PSF). In fact, let $\p\in S_{X^*}$ be such that $F=B_X\cap\{\langle\p,\cdot\rangle=1\}$ and choose any $x\in F$. Then, for every $y\in Y:=\ker\p$ with $\|y\|=1$ we have $\|x+y\|>1$ (otherwise, $x+y\in F$, whence ${\rm diam}(F)\geq\|(x+y)-x\|=1$). Consequently, the following proposition is an immediate consequence of the previous result and the simple fact that if $Y$ is a subspace of $X$, then every face of $B_Y$ is contained in a face of $B_X$.
\begin{proposition} Let $X$ be an infinite-dimensional Banach space. Suppose that, for every infinite-dimensional subspace $Y$ of $X$, the unit ball $B_Y$ contains a face with diameter strictly smaller than $1$. Then the unit sphere of $X$ contains a symmetrically $(1+)$-separated subset with cardinality $w^*\text{-}{\rm dens}\, X^*$.

\smallskip
In particular, the conclusion holds true if every face of $X$ has diameter strictly smaller than $1$.
\end{proposition}

Clearly, if $X$ is strictly convex, then every face of $B_X$ is a singleton; we can therefore record the following immediate consequence to the present proposition.
\begin{corollary}\label{strictly convex gives (1+)} If $X$ is an infinite-dimensional strictly convex Banach space, then the unit sphere of $X$ contains a symmetrically $(1+)$-separated subset with cardinality $w^*\text{-}{\rm dens}\, X^*$.
\end{corollary}

Our next application of Theorem \ref{general th for (1+)} leads us back to the dawning of this section, \emph{i.e.}, to exposed points. In fact, if $x\in S_X$ is an exposed point for $B_X$ and $\p\in S_{X^*}$ exposes $x$, then, for every unit vector $y\in Y:=\ker\p$, we have $\langle\p,x+y\rangle=1$, whence $\|x+y\|>1$. Consequently, if the unit ball of $X$ admits an exposed point, $X$ has (PSF).\smallskip

We now need to recall that the result by Lindenstrauss and Troyanski the proof of Theorem \ref{(1+) in reflexive} heavily relied on is in fact consequence of a more general result, due to Phelps (\cite{Phelps}, \emph{cf.} \cite[Theorem 5.17]{BeLi}): if $C$ is a closed, convex and bounded set with the Radon-Nikodym property (RNP) in a Banach space $X$, then $C$ is the closed convex hull of its strongly exposed points. This assertion is truly more general, since every weakly compact convex set has the RNP (\cite[Theorem 5.11.(i)]{BeLi}), while the unit ball of $\ell_1$ is a~simple example of a~set with the RNP that fails to be weakly compact. Therefore, if a~Banach space $X$ has the RNP, the unit ball of every its subspace is the closed convex hull of its strongly exposed points. As a consequence, every infinite-dimensional subspace of $X$ has (PSF) and we infer the validity of the following result. Note that, of course, it contains Theorem \ref{(1+) in reflexive} as a particular case.
\begin{theorem} Let $X$ be an infinite-dimensional Banach space with the Radon--Nikodym property. Then there exists a symmetrically $(1+)$-separated family of unit vectors in $X$, with cardinality $w^*\text{-}{\rm dens}\, X^*$.
\end{theorem}

Our last result in this chapter is devoted to duals to G\^ateaux differentiability spaces (we refer to Section \ref{review exposed} for the definition) and exploits $w^*$-exposed points. Note that, if $X$ is a G\^ateaux differentiability space, there exists a functional $f_1\in S_{X^*}$ which is $w^*$-exposed by some $x_1\in S_X$; in particular, for every $g\in \{x_1\}^\perp$ we have $\|f_1\pm g\|>1$. The $w^*$-closed subspace $\{x_1\}^\perp$ of $X^*$ is the dual to $X/{\rm span}\,x_1$, which is a G\^ateaux differentiability space, according to \cite[Proposition 6.8]{Phelps diff}; therefore, we may use a transfinite induction argument completely analogous to those already presented in this section. Let us record the result formally.

\begin{proposition} Let $X$ be a Banach space dual to a G\^ateaux differentiability space. Then the unit sphere of $X$ contains a symmetrically $(1+)$-separated subset with cardinality $w^*\text{-}{\rm dens}\, X^*$.
\end{proposition}

\subsection{$(1+\e)$-separation}
In this part we shall present results parallel to those of the previous section, but concerning the existence of large symmetrically $(1+\e)$-separated families. It is clear that one could formulate a uniform analogue to condition (PSF) and adapt the arguments to be presented in this section to deduce an analogue to Theorem \ref{general th for (1+)}. However, we shall not pursue this direction and we shall restrict ourselves to the consideration of some classes of Banach spaces.\smallskip

In the first result we shall exploit the full power of the notion of strongly exposed point in order to treat spaces with the RNP. We start with the following simple observation.
\begin{lemma}\label{Lemma: strong exp gives (1+e)} Let $X$ be a Banach space and $x\in B_X$ be a strongly exposed point of $B_X$; also let $\p \in X^*$ be a strongly exposing functional for $x$. Then 
$$\inf \{ \|x+v\|:v\in \ker \p, \|v\|=1 \} >1.$$
\end{lemma}
\begin{proof} Note preliminarily that the above infimum is necessarily greater or equal to 1. In fact, $\p$ exposes $x$, so for every non-zero $v \in\ker \p$ we have $\|x+v\|>1$. If by contradiction the conclusion of the lemma is false, we may find a sequence of unit vectors $(v_n)_{n=1}^\infty$ in $\ker\p$  such that $\|x+v_n\|\rightarrow 1$. The vectors $r_n:=\frac{x+v_n}{\|x+v_n\|}\in B_X$ then satisfy $\langle\p,r_n\rangle=\frac{\langle\p,x\rangle}{\|x+v_n\|}\rightarrow\langle\p,x\rangle$; consequently, our assumption that $x$ is strongly exposed by $\p$ allows us to conclude that $r_n\rightarrow x$. This is however an absurdity, since
$$\|r_n-x\| \geq \|x+v_n -x\| - \|r_n -(x+v_n)\| = 1 - \|x+v_n\| \cdot \left|{\frac{1}{\|x+v_n\|} -1}\right| \rightarrow 1.$$
\end{proof}

\begin{theorem}\label{Thm: (1+e) for RNP} Let $X$ be an infinite-dimensional Banach space with the RNP and let $\kappa\leq w^*\text{-}{\rm dens}\, X^*$ be a cardinal number with uncountable cofinality. Then, for some $\e>0$, the unit sphere of $X$ contains a symmetrically $(1+\e)$-separated subset, with cardinality $\kappa$.
\end{theorem}
\begin{proof} The argument follows a pattern similar to the proofs of the previous section, therefore we only sketch it. Let $\lambda:=w^*\text{-}{\rm dens}\, X^*$; a transfinite induction argument as in the proof of Theorem \ref{(1+) in reflexive} shows the existence of families of closed subspaces $(X_\alpha)_{\alpha<\lambda}$ of $X$, unit vectors $(x_\alpha)_{\alpha<\lambda}$ and functionals $(\p_\alpha)_{\alpha<\lambda}$ with the following properties, for every $\alpha<\lambda$:
\begin{romanenumerate}
\item $X_\alpha=\bigcap_{\gamma<\alpha}\ker\p_\gamma$;
\item $x_\alpha\in X_\alpha$ is a strongly exposed point for $B_{X_\alpha}$ in $X_\alpha$, strongly exposed by $\p_\alpha$. 
\end{romanenumerate}
According to Lemma \ref{Lemma: strong exp gives (1+e)}, we can also find, for each $\alpha<\lambda$, a real $\e_\alpha>0$ such that $\|x_\alpha+v\|\geq1+\e_\alpha$ for every unit vector $v\in\ker\p_\alpha\cap X_\alpha$. In particular, for every $\alpha<\beta<\lambda$ we have $\pm x_\beta\in\ker\p_\alpha\cap X_\alpha$, whence $\|x_\alpha\pm x_\beta\|\geq1+\e_\alpha$.\smallskip

We finally exploit the cofinality of $\kappa$ to conclude the proof. Of course, the union of the sets $\Gamma_n:=\{\alpha<\kappa\colon\e_\alpha\geq 1/n\}$ covers $\kappa$, whence the uncountable cofinality of $\kappa$ implies the existence of $n_0$ such that $|\Gamma_{n_0}|=\kappa$. Consequently, for any $\alpha,\beta\in\Gamma_{n_0}, \alpha<\beta$, we have
$$\|x_\alpha\pm x_\beta\|\geq1+\e_\alpha\geq1+ 1/n_0.$$
Therefore, the family $\{x_\alpha\}_{\alpha\in\Gamma_{n_0}}$ has cardinality $\kappa$ and it is symmetrically $(1+1/n_0)$-separated, which concludes the proof.
\end{proof}

Plainly, reflexive Banach spaces have the RNP and satisfy ${\rm dens}\, X=w^*\text{-}{\rm dens}\, X^*$. Therefore, the following corollary is a particular case of the previous theorem; note that, by the considerations at the beginning of Section \ref{Section: geometry (1+)}, it implies (ii) in Theorem B.
\begin{corollary}\label{(1+e) in reflexive} Suppose that $X$ is an infinite-dimensional reflexive Banach space. Let $\kappa\leq{\rm dens}\, X$ be a cardinal number with uncountable cofinality. Then there exists a symmetrically $(1+\e)$-separated family of unit vectors in $S_X$, with cardinality $\kappa$.
\end{corollary}

We next give the $(1+\e)$-separation analogue to Corollary \ref{strictly convex gives (1+)}; as it is to be expected, we need to assume a uniform analogue to strict convexity. Let us recall that a norm $\n$ on a Banach space $X$ is \emph{locally uniformly rotund} (hereinafter, \emph{LUR}) if, for every $x\in S_X$ and every sequence $(x_n)_{n=1}^\infty$ in $S_X$, the condition $\|x_n+x\|\to2$ implies $x_n\to x$. It is very easy to verify the standard fact that if the norm of $X$ is LUR, then every point of $S_X$ is a~strongly exposed point for $B_X$.\smallskip

Consequently, by following the same pattern as in the proof of Theorem \ref{Thm: (1+e) for RNP}, we may conclude the following result.
\begin{proposition} Let $X$ be an infinite-dimensional Banach space and let $\kappa\leq w^*\text{-}{\rm dens}\,X^*$ be a cardinal number with ${\rm cf}(\kappa)$ uncountable. If $X$ is LUR, then for some $\e>0$ the unit sphere of $X$ contains a symmetrically $(1+\e)$-separated subset, with cardinality $\kappa$.
\end{proposition}

\begin{remark} It follows from a well-known renorming result, due to Troyanski and Zizler (\cite{Troyanski,Zizler}, also see \cite[Section 7.1]{DGZ}), that every Banach space with a projectional skeleton admits a LUR renorming; we shall not recall the definition of a projectional skeleton here and we refer the reader to \cite{Kubis} or \cite[Corollary 17.5]{KKLP} for information on this topic. As a consequence, when we combine this result with our previous proposition, we obtain the existence of a renorming whose unit sphere contains a large symmetrically $(1+\e)$-separated family, with no need to exploit biorthogonal systems, as we made in Section \ref{Section: Auerbach}.
\end{remark}

The last result for this section is dedicated to asymptotically uniformly convex Banach spaces (\emph{cf.} Section \ref{review AUC}); the study of separated sequences in those spaces was undertaken by Delpech, \cite{delpech}. Our result will, in particular, provide a non-separable counterpart to Deplech's contribution.

\begin{theorem} Let $X$ be an infinite-dimensional Banach space and let $d<1+\overline{\delta}_X(1)$. Then the unit sphere of $X$ contains a symmetrically $d$-separated family with cardinality equal to $w^*\text{-}{\rm dens}\, X^*$.
\end{theorem}
We should perhaps mention that the result is of interest only for asymptotically uniformly convex Banach spaces, or, more generally, whenever $\overline{\delta}_X(1)>0$. In fact, for $d<1$, the existence of a symmetrically $d$-separated subset of $S_X$, with cardinality ${\rm dens}\, X$, is an~immediate consequence of Riesz' lemma.

\begin{proof} Let $\lambda:=w^*\text{-}{\rm dens}\, X^*$. We construct by transfinite induction a long sequence of unit vectors $(x_\alpha)_{\alpha<\lambda}\subseteq S_X$ and a decreasing long sequence $(H_\alpha)_{\alpha<\lambda}$ of infinite-dimensional subspaces of $X$ with $\dim(H_\alpha/H_{\alpha+1})<\infty$, $\dim(X/H_1)<\infty$ and such that:
\begin{romanenumerate}
\item $\|x_\alpha+h\|\geq d$, for each $h\in S_{H_\alpha}$ and $\alpha<\lambda$;
\item $x_\beta\in H_\alpha$, for $\alpha<\beta<\lambda$.
\end{romanenumerate}
It immediately follows that, for $\alpha<\beta$, $\pm x_\beta\in H_\alpha$, whence $\|x_\alpha\pm x_\beta\|\geq d$ and we are done.

We start with an arbitrary $x_0\in S_X$; since $d-1<\auc_X(1,x_0)$ there exists a finite-codimensional subspace $H_0$ of $X$ such that $\inf_{h\in H_0, \|h\|\geq1}\|x_0+h\|\geq d$. In particular, $\|x_0+h\|\geq d$ for every $h\in S_{H_0}$. We now choose arbitrarily $x_1\in S_{H_0}$ and we proceed by transfinite induction. Assuming to have already found $(x_\alpha)_{\alpha<\gamma}$ and $(H_\alpha)_{\alpha<\gamma}$ for some $\gamma<\lambda$ we consider two cases: if $\gamma=\tilde{\gamma}+1$ is a successor ordinal, we can choose arbitrarily $x_\gamma\in S_{H_{\tilde{\gamma}}}$. According to Fact \ref{Fact: AUC modulus of subspace}, we have $d-1<\auc_X(1)\leq \auc_{H_{\tilde{\gamma}}}(1)\leq\auc_{H_{\tilde{\gamma}}}(1,x_\gamma)$; consequently, we may find a finite-codimensional subspace $H_\gamma$ of $H_{\tilde{\gamma}}$ such that $\|x_\gamma +h\|\geq d$, for every $h\in S_{H_\gamma}$.\smallskip

In the case that $\gamma$ is a limit ordinal, we first note that $\cap_{\alpha<\gamma} H_\alpha$ is infinite-dimensional. In fact, each $H_{\alpha+1}$ is the intersection of $H_\alpha$ with the kernels of finitely many functionals; therefore, $\cap_{\alpha<\gamma} H_\alpha$ is the intersection of the kernels of at most $|\gamma|<\lambda$ functionals. Consequently, the minimal cardinality of a family of functionals that separates points in $\cap_{\alpha<\gamma}H_\alpha$ is $\lambda$ and, in particular, $\cap_{\alpha<\gamma} H_\alpha$ is infinite-dimensional. We can therefore choose a norm-one $x_\gamma\in\cap_{\alpha<\gamma}H_\alpha$ and, arguing as above, we find $H_\gamma\subseteq\cap_{\alpha<\gamma}H_\alpha$ such that $\dim(\cap_{\alpha<\gamma} H_\alpha/H_\gamma)<\infty$ and $\|x_\alpha+h\|\geq d$ for $h\in S_{H_\gamma}$. This completes the inductive step and consequently the proof.
\end{proof}

If we combine the above theorem with the inequality $\delta_X(1)\leq\overline{\delta}_X(1)$, we arrive at the following particular case to claim (iii) in Theorem B.
\begin{corollary} Let $X$ be a uniformly convex Banach space. Then for every $\e>0$, the unit sphere of $X$ contains a symmetrically $(1+\delta_X(1)-\e)$-separated family with cardinality ${\rm dens}\, X$.
\end{corollary}

\subsection{Injections into $\ell_p(\Gamma)$}\label{injections lp(gamma)}
By a result of Lindenstrauss (\cite{lin}), every non-separable reflexive Banach space has a projectional resolution of the identity (PRI). The following idea--already exploited to some extent in the proof of \cite[Theorem 3.5]{KaKo}--is essentially due to Benyamini and Starbird (\cite[p.~139]{benyaministarbird}) who, using the Gurarii--James inequality (\cite{james}), noticed that if $X$ is super-reflexive, then for any $\e>0$ and every PRI $(P_\alpha)_{\omega\leqslant \alpha\leqslant\lambda}$ in $X$, where $\lambda = {\rm dens}\, X$, there exists $p\in (1,\infty)$ such that the operator
\begin{equation}\label{operator}T\colon X \;\longrightarrow\;\; \Biggr( \bigoplus_{\omega\leqslant \alpha<\lambda} (P_{\alpha+1}-P_\alpha)(X) \Biggr)_{\!\!\ell_p\left([\omega, \lambda) \right)}\end{equation}
given by
$$Tx = \bigl(P_{\alpha+1}x - P_{\alpha}x\bigr)_{\omega\leqslant \alpha<\lambda}\qquad (x\in X)$$
has norm at most $2+\e$. Let us record here the following fact.
\begin{proposition}Let $X$ be a super-reflexive space $X$ and let $\lambda = {\rm dens}\, X$. Then for every $\e>0$ there exist $p\in(1,\infty)$, a closed subspace $Y$ of $X$ with ${\rm dens}(Y)=\lambda$, and a~linear injection $S\colon Y\to \ell_p(\lambda)$ with $\|S\|\leqslant 2+\varepsilon$ that maps some family of unit vectors $(y_\alpha)_{\alpha<\lambda}$ in $Y$ onto the unit vector basis of $\ell_p(\lambda)$.\end{proposition}
\begin{proof}The case where $X$ is separable follows from the Gurarii--James inequality almost directly. Let us consider the case where $X$ is non-separable.\smallskip

Fix $\varepsilon > 0$ and a PRI $(P_\alpha)_{\omega\leqslant \alpha\leqslant\lambda}$ in $X$. Take the corresponding operator $T$ given by \eqref{operator}. For every $\alpha \in [\omega, \lambda)$ we pick a unit vector $y_\alpha$ in $(P_{\alpha+1}-P_\alpha)(X)$. Then $\|Ty_\alpha\|=1$ and the closed linear span in the codomain of $T$ of the set $\{Ty_\alpha\colon \alpha \in [\omega, \lambda)\}$ is isometric to $\ell_p(\lambda)$. Consequently, $Y$ being the closed linear span of $\{y_\alpha\colon \alpha \in [\omega, \lambda)\}$ is the desired subspace of $X$ and for $S$ we simply take the restriction of $T$ to $Y$.
\end{proof}

\begin{remark}As proved by the first-named author (\cite{hajek}), the hypothesis of super-reflexivity cannot be relaxed to mere reflexivity of the space $X$.\end{remark}

Consider the net $\mathscr{L}$ comprising subsets of $\lambda$ whose complements have cardinality less than $\lambda$, ordered with the reversed inclusion. In a Banach space $X$ fix a family of unit vectors $\{y_\beta\colon \beta <\lambda\}$. For every set $\Lambda\subset \lambda$ we define $X_\Lambda$ to be the closed linear span of the set $\{y_\beta\colon \beta \in \Lambda\}$. We say that an operator $S\colon X\to \ell_p(\lambda)$ is \emph{bounded by a pair} $(\varpi, \varrho)$, when $\|S\|\leqslant \varrho$ and $\|S|_{X_\Lambda}\|\geqslant \varpi$ for every $\Lambda\in \mathscr{L}$.

\begin{proposition}\label{lowerq}Let $X$ be a Banach space and let $\lambda = {\rm dens}\, X$. Suppose that there exists a bounded linear injection $S\colon X\to \ell_p(\lambda)$ that maps some collection $(x_\alpha)_{\alpha<\lambda}$ of unit vectors in $X$ onto the standard unit vector basis of $\ell_p(\lambda)$. Then for every $\varepsilon >0$, $S_{X}$ contains a~symmetrically $(2^{1/p}-\e)$-separated subset of cardinality $\lambda$.\end{proposition}

\begin{proof}For every $\Lambda \subset \lambda$, we set $\varrho_\Lambda=\|S\vert_{X_\Lambda}\|$. As $\|Sy_\alpha\|=\|y_\alpha\|=1$ ($\alpha<\lambda$), we have $\varrho_\Lambda \geqslant 1$. The net $(\varrho_\Lambda)_{\Lambda\in \mathscr{L}}$ is then non-increasing (let us recall that $\mathscr{L}$ is endowed with the reversed inclusion). For every $\Lambda \in \mathscr{L}$, the restriction operator $S|_{X_\Lambda}$ is bounded by the pair $$(\inf_{\Upsilon \in  \mathscr{L}}\varrho_\Upsilon,\varrho_\Lambda).$$ Moreover, $\varrho_\Lambda\to \inf_{\Upsilon \in  \mathscr{L}}\varrho_\Upsilon$ as $\Lambda \searrow \varnothing$. We may thus always replace $X$ with $X_\Lambda$ for a~small enough set $\Lambda \in \mathscr L$ and assume that $S\colon X\rightarrow\ell_p(\lambda)$ is bounded by a pair $(\varpi,\varrho)$ with $\frac{\varpi}{\varrho}$ as close to $1$ as we wish (in which case $\frac{\varpi}{\varrho}<1$).\smallskip

Take $\tilde{\varpi} \in (0,\varpi)$. Since $\|S\| >\tilde{\varpi}$, we may find a unit vector $y_1$ in $X_\lambda$ such that $\|Sy_1\| > \tilde{\varpi}$; up to a small perturbation, we may also assume that $Sy_1$ is finitely supported. Take a non-zero $\alpha<\lambda$ and assume that we have already found unit vectors $y_\beta$ ($\beta<\alpha$) in $X_\lambda$ such that
\begin{romanenumerate}
\item $\|Sy_\beta\| > \tilde{\varpi}$ ($\beta<\alpha$),
\item ${\rm supp}\,Sy_\beta$ is a finite subset of $\lambda$ ($\beta<\alpha$), and
\item ${\rm supp}\, Sy_{\beta_1}\cap\, {\rm supp}\, Sy_{\beta_2}=\varnothing$ for distinct $\beta_1,\beta_2<\alpha$.
\end{romanenumerate}
According to (ii), the set $$\Lambda_\alpha = \bigcup_{\beta<\alpha}{\rm supp}\, Sy_{\beta}$$ has cardinality $|\Lambda_\alpha|<\lambda$, that is, $\lambda\setminus \Lambda_\alpha\in \mathscr L$. Consequently, we may find a unit vector $y_\alpha\in X_{\lambda \setminus \Lambda_\alpha}$ such that ${\rm supp}\, Sy_{\alpha}$ is a finite subset of $\lambda$, disjoint from $\Lambda_\alpha$ and $\|S y_\alpha\| > \tilde{\varpi}$.\smallskip

As the vectors $(Sy_\alpha)_{\alpha<\lambda}$ are pairwise disjointly supported, for distinct $\beta_1,\beta_2<\lambda$ we have

$$\varrho \cdot \|y_{\beta_1} \pm y_{\beta_2}\| \geq \| Sy_{\beta_1} \pm Sy_{\beta_2}\| = \left( \|Sy_{\beta_1}\|^p + \|Sy_{\beta_2}\|^p \right)^{1/p} \geq \tilde{\varpi} \cdot 2^{1/p}.
$$
Since $\frac{\tilde{\varpi}}{\varrho}$ may be chosen to be as close to $1$ as we wish, the proof is finished.
\end{proof}

To conclude, it is plain that the conjunction of the two propositions presented in this section implies clause (iii) in Theorem B.

\section{A renorming of $c_0(\omega_1)$ with no uncountable Auerbach systems}\label{No uncountable Auerbach}
The main result of this section is the construction of a WLD Banach space with density character $\omega_1$ that contains no uncountable Auerbach system. In particular, we wish to construct a renorming of the real space $c_0(\omega_1)$ such that in this new norm the space fails to contain uncountable Auerbach systems. This is indeed possible, at least under the assumption of the Continuum Hypothesis.
\begin{theorem}[CH] There exists a renorming $\nn\cdot$ of the real Banach space $c_0(\omega_1)$ such that the space $(c_0(\omega_1),\nn\cdot)$ contains no uncountable Auerbach systems.
\end{theorem}

The subsequent results present in this section are entirely devoted to the construction of the desired norm and the verification of the asserted property. The section is therefore naturally divided into two parts: in the former one, we introduce a family of equivalent norms (depending on some parameters) on $c_0(\omega_1)$ and prove some properties of every such norm; in the latter, we prove that careful choice of the parameters implies that the resulting space contains no uncountable Auerbach system. Throughout the section, we assume real scalars.\smallskip

We shall denote by $\n_\infty$, or just $\n$ if no confusion may arise, the canonical norm on $c_0(\omega_1)$ and by $(e_\alpha)_{\alpha<\omega_1}$ its canonical long Schauder basis; the corresponding set of biorthogonal functionals, in $\ell_1(\omega_1)=c_0(\omega_1)^*$, will be denoted by $(e^*_\alpha)_{\alpha<\omega_1}$.\smallskip

Fix a parameter $\delta>0$ so small that $\Delta:=\frac{\delta}{1-\delta}\leq 1/5$ (for example, we could choose $\delta=1/6$). We also select an injective long sequence $(\lambda_\alpha)_{\alpha<\omega_1}\subseteq(0,\delta)$. Moreover, for every $\alpha<\omega_1$ there exists an enumeration $\sigma_\alpha$ of the set $[0,\alpha)$, \emph{i.e.}, a bijection $\sigma_\alpha:|\alpha|\to\alpha$ (observing that $|\alpha|$ is either $\omega$ or a finite cardinal). We may therefore assume to have selected, for every $\alpha<\omega_1$, a fixed bijection $\sigma_\alpha$. Having fixed such notation, we are now in position to define elements $\p_\alpha\in\ell_1(\omega_1)$ ($\alpha<\omega_1$) as follows: $\p_0=e^*_0$ and, for $1\leq\alpha<\omega_1$,
$$\p_\alpha(\eta)=\begin{cases} 1&\textrm{if}\;\eta=\alpha\\ 0&\textrm{if}\;\eta>\alpha\\(\lambda_\alpha)^k & \textrm{if}\;\eta<\alpha,\,\eta=\sigma_\alpha(k).\end{cases}$$
The above enumerations may be chosen arbitrarily and the subsequent argument will not depend on any specific such choice. On the other hand, a substantial part of the argument to be presented will consist in explaining how to properly choose the coefficients $\lambda_\alpha$.\smallskip

We start with a few elementary properties of the functionals $(\p_\alpha)_{\alpha<\omega_1}$ and their use in the definition of the renorming. Plainly, $\|\p_\alpha-e^*_\alpha\|_1=\sum_{k=1}^{|\alpha|}(\lambda_\alpha)^k\leq \sum_{k=1}^\infty \delta^k =\Delta$, whence it follows that $\p_\alpha\in c_0(\omega_1)^*$ and $\|\p_\alpha\|_1\leq 1+\Delta$.

\begin{fact}\label{fact: basis equiv to ell_1} $(\p_\alpha)_{\alpha<\omega_1}$ is a long Schauder basis for $\ell_1(\omega_1)$, equivalent to the canonical Schauder basis $(e^*_\alpha)_{\alpha<\omega_1}$ of $\ell_1(\omega_1)$.
\end{fact}
\begin{proof} From the inequality $\|\p_\alpha - e^*_\alpha\|_1\leq\Delta$ observed above, it follows that
$$\left\Vert\sum_{i=1}^n d_i\p_{\alpha_i}\right\Vert_1\geq \left\Vert\sum_{i=1}^n d_i e^*_{\alpha_i}\right\Vert_1-\left\Vert\sum_{i=1}^n d_i(e^*_{\alpha_i}-\p_{\alpha_i})\right \Vert_1\geq$$
$$\geq\sum_{i=1}^n|d_i|-\sum_{i=1}^n|d_i|\|\p_{\alpha_i} - e^*_{\alpha_i}\|_1\geq(1-\Delta)\sum_{i=1}^n|d_i|$$
for every choice of scalars $(d_i)_{i=1}^n$. Combining this inequality with $\|\p_\alpha\|_1\leq 1+\Delta$ results in
$$(1-\Delta)\sum_{\alpha<\omega_1}|d_\alpha| \leq \left\Vert\sum_{\alpha<\omega_1} d_\alpha\p_{\alpha}\right\Vert_1 \leq (1+\Delta)\sum_{\alpha<\omega_1}|d_\alpha|.$$

In order to prove that $\overline{\rm span}\{\p_\alpha\}_{\alpha<\omega_1}=\ell_1(\omega_1)$, we shall consider the bounded linear operator $T\colon \ell_1(\omega_1)\to\ell_1(\omega_1)$ such that $T(e^*_\alpha)=\p_\alpha$. The first part of the argument shows that $T$ is, indeed, a bounded linear operator and that $\|T-I\|\leq\Delta<1$. Consequently, $T$ is an isomorphism of $\ell_1(\omega_1)$ onto itself, whence the closed linear span of $(\p_\alpha)_{\alpha<\omega_1}$ equals $\ell_1(\omega_1)$.
\end{proof}

We may now exploit the functionals $(\p_\alpha)_{\alpha<\omega_1}$ to define a renorming of $c_0(\omega_1)$.
\begin{definition}
$$\nn{x}:=\sup_{\alpha<\omega_1}|\langle\p_\alpha,x\rangle|\qquad (x\in c_0(\omega_1)).$$
Moreover, we also denote by $X$ the space $X:=(c_0(\omega_1),\nn{\cdot})$.
\end{definition}
Let us preliminarily note that if $\alpha<\beta$, then there exists $k\in\N$ such that $\alpha=\sigma_\beta(k)$ and consequently $|\langle\p_\beta,e_\alpha\rangle|=(\lambda_\beta)^k\leq 1$; it immediately follows that $\nn{e_\alpha}=1$. We may then readily check that $\nn{\cdot}$ is a norm, equivalent to $\n_\infty$. In fact, the inequality $\nn{\cdot}\leq(1+\Delta)\n_\infty$ is obvious and the lower estimate follows from a familiar pattern: if $\gamma<\omega_1$ is such that $\|x\|_\infty=|x(\gamma)|$, we then have
$$\nn{x}=\nn{2x(\gamma)e_\gamma+x-2x(\gamma)e_\gamma}\geq 2|x(\gamma)|-\nn{x-2x(\gamma)e_\gamma}$$
$$\geq2\|x\|_\infty-(1+\Delta)\|x-2x(\gamma)e_\gamma\|_\infty=(1-\Delta)\|x\|_\infty.$$

We shall also denote by $\nn{\cdot}$ the dual norm on $\ell_1(\omega_1)$; needless to say, such a  norm satisfies $(1+\Delta)^{-1}\n_1\leq\nn{\cdot}\leq(1-\Delta)^{-1}\n_1$. The definition of $\nn\cdot$ clearly implies that $\nn{\p_\alpha}\leq 1$ and $\langle\p_\alpha,e_\alpha\rangle=1$ actually forces $\nn{\p_\alpha}=1$.\smallskip

By the very definition, $\{\p_\alpha\}_{\alpha<\omega_1}$ is a $1$-norming set for $X$; we now note the important fact that such a collection of functionals is actually a boundary for $X$. A similar argument also shows that $X$ is a polyhedral Banach space and we also record such information in the next lemma, even if we shall not need it in what follows.
\begin{lemma} $\{\p_\alpha\}_{\alpha<\omega_1}$ is a boundary for $X$. Moreover, $X$ is a polyhedral Banach space.
\end{lemma}
\begin{proof} In the proof of the first assertion, by homogeneity, it is clearly sufficient to consider $x\in(c_0(\omega_1),\nn\cdot)$ with $\|x\|_\infty=1$; in particular, there is $\alpha<\omega_1$ with $|x(\alpha)|=1$. For such an $\alpha$ we thus have:
$$|\langle\p_\alpha,x\rangle|\geq|\langle e^*_\alpha,x\rangle|-|\langle\p_\alpha-e^*_\alpha,x\rangle|\geq 1-\|\p_\alpha-e^*_\alpha\|_1\|x\|_\infty\geq 1-\Delta.$$
On the other hand, consider the finite set $N_x:=\{\gamma<\omega_1\colon|x(\gamma)|>\Delta\}$. Then for each $\alpha\notin N_x$ we obtain:
$$|\langle\p_\alpha,x\rangle|\leq|\langle e^*_\alpha,x\rangle|+|\langle\p_\alpha-e^*_\alpha,x\rangle|\leq |x(\alpha)|+\Delta\leq 2\Delta\leq 1-\Delta.$$
Consequently, the supremum appearing in the definition of $\nn{x}$ is actually over the finite set $N_x$ and it is therefore attained.\smallskip

We then turn to the polyhedrality of $(c_0(\omega_1),\nn{\cdot})$. Let $E$ be any finite-dimensional subspace of $(c_0(\omega_1),\nn{\cdot})$ and let $x_1,\dots,x_n$ be a finite $\Delta/2$-net (relative to the $\n_\infty$ norm) for the set $\{x\in E\colon\|x\|_\infty =1\}$. Let us consider the finite set $N_E:=\cup_{i=1}^nN_{x_i}$, where $N_x:=\{\gamma<\omega_1\colon|x(\gamma)|>\Delta/2\}$; then for every $x\in E$ with $\|x\|_\infty=1$ there clearly holds $$\{\gamma<\omega_1\colon|x(\gamma)|>\Delta\}\subseteq N_E.$$ The same calculations as before then demonstrate that for all such $x$ we have $$\sup_{\alpha<\omega_1}|\langle\p_\alpha,x\rangle|= \max_{\alpha\in N_E}|\langle\p_\alpha,x\rangle|.$$ Consequently, $\nn\cdot=\max_{\alpha\in N_E}|\langle\p_\alpha,\cdot\rangle|$ on $E$ and $\{\p_\alpha\}_{\alpha\in N_E}$ is a finite boundary for $E$, which is thus polyhedral.
\end{proof}

We now turn to the first crucial result for what follows, namely the fact that every norm-attaining functional on $X$ is finitely supported with respect to the basis $(\p_\alpha)_{\alpha<\omega_1}$.
\begin{theorem}\label{norm-attain finitely supported} Let $g\in S_{X^*}$ be a norm-attaining functional and let $u\in S_X$ be such that $\langle g,u\rangle=1$. Also, denote by $F$ the finite set $F:=\{\alpha<\omega_1\colon |u(\alpha)|>\frac{\Delta}{1-\Delta}\}$. Then
$$g=\sum_{\alpha\in F}g_\alpha\p_\alpha\qquad \text{and} \qquad\sum_{\alpha\in F}|g_\alpha|\leq1.$$
\end{theorem}
\begin{proof} We begin with two very simple remarks, that we shall use in the course of the argument. Combining the estimate in the proof of Fact \ref{fact: basis equiv to ell_1} with $\nn{\p_\alpha}\leq1$, we readily deduce that
\begin{equation}\label{eqz: equiv of basis in |||}
\frac{1-\Delta}{1+\Delta}\sum_{\alpha<\omega_1}|d_\alpha|\leq\nn{\sum_{\alpha<\omega_1}d_\alpha\p_\alpha}\leq\sum_{\alpha<\omega_1}|d_\alpha|.
\end{equation}
Moreover, we have 
\begin{equation}\label{eqz: w* closure}
\overline{{\rm conv}}^{w^*}\left\{\pm\p_\alpha\right\}_{\alpha<\omega_1}=B_{X^*}.
\end{equation}
This is an immediate consequence of the Hahn--Banach separation theorem and the fact that $\{\pm\p_\alpha\}_{\alpha<\omega_1}$ is $1$-norming for $X$.\smallskip

We now start with the argument: let $g\in S_{X^*}$ be a norm-attaining functional and choose $u\in S_X$ such that $\langle g,u\rangle=1$. We may express $g$ as $g=\sum_{\alpha<\omega_1}g_\alpha\p_\alpha$, and we also set $F:=\{\alpha<\omega_1\colon |u(\alpha)|>\frac{\Delta}{1-\Delta}\}$. Moreover, let us choose arbitrarily $$f=\sum_{\alpha<\omega_1}d_\alpha\p_\alpha\in{\rm conv}\left\{\pm\p_\alpha\right\}_{\alpha<\omega_1}$$ (\emph{i.e.}, only finitely many $d_\alpha$'s are non-zero and $\sum_{\alpha<\omega_1}|d_\alpha|\leq1$).
\begin{claim*} If $\langle f,u\rangle\geq 1-\eta$, then $$\sum_{\alpha\notin F}|d_\alpha|\leq12\eta.$$
\end{claim*}
\begin{proof}[Proof of the Claim.] Recall that $(1+\Delta)^{-1}\leq\|u\|_\infty\leq(1-\Delta)^{-1}$; hence for each $\alpha\notin F$ we have
$$|\langle\p_\alpha,u\rangle|\leq|\langle e^*_\alpha,u\rangle|+|\langle\p_\alpha-e^*_\alpha,u\rangle|\leq\frac{\Delta}{1-\Delta}+\Delta\|u\|_\infty\leq2\frac{\Delta}{1-\Delta}.$$ 
Consequently, setting $\rho:=\sum_{\alpha\notin F}|d_\alpha|$, we have
$$1-\eta\leq\langle f,u\rangle\leq\sum_{\alpha\notin F}|d_\alpha| |\langle\p_\alpha,u\rangle|+\sum_{\alpha\in F}d_\alpha\langle\p_\alpha,u\rangle\leq 2\rho\frac{\Delta}{1-\Delta}+\sum_{\alpha\in F}d_\alpha\langle\p_\alpha,u\rangle,$$
whence
$$(\dagger)\qquad\sum_{\alpha\in F}d_\alpha\langle\p_\alpha,u\rangle\geq1-\eta-2\rho\frac{\Delta}{1-\Delta}.$$

On the other hand, there exists $\overline{\alpha}<\omega_1$ with $|u(\overline{\alpha})| \geq(1+\Delta)^{-1}$ and for such $\overline{\alpha}$ we have
$$(\ddagger)\qquad|\langle\p_{\overline{\alpha}},u\rangle|\geq|\langle e^*_{\overline{\alpha}},u\rangle|- |\langle\p_{\overline{\alpha}}-e^*_{\overline{\alpha}},u\rangle|\geq\frac{1}{1+\Delta}-\frac{\Delta}{1-\Delta}.$$

Let us now consider the functional
$$\psi:=\sum_{\alpha\in F}d_\alpha\p_\alpha+\rho\cdot{\rm sgn}\langle \p_{\overline{\alpha}},u\rangle\p_{\overline{\alpha}};$$
clearly, $\psi\in{\rm conv}\{\pm\p_\alpha\}_{\alpha<\omega_1}$, so $\nn{\psi}\leq1$. We can therefore combine this information with $(\dagger)$ and $(\ddagger)$ and conclude that 
\begin{eqnarray*}
1\geq\langle\psi,u\rangle &=& \sum_{\alpha\in F}d_\alpha\langle\p_\alpha,u\rangle+ \rho\cdot|\langle\p_{\overline{\alpha}},u\rangle|\\
&\geq& 1-\eta-2\rho\frac{\Delta}{1-\Delta}+\rho\cdot\left(\frac{1}{1+\Delta}- \frac{\Delta}{1-\Delta}\right)\\
&=& 1-\eta+\rho\cdot\left(\frac{1}{1+\Delta}-3\frac{\Delta}{1-\Delta}\right)\geq 1-\eta+\frac{\rho}{12};
\end{eqnarray*}
in the last inequality we used the fact that $\Delta\mapsto\left(\frac{1}{1+\Delta}- 3\frac{\Delta}{1-\Delta}\right)$ is a decreasing function on $(0,1)$ and the assumption that $\Delta\leq1/5$. It follows that $\rho/12\leq\eta$, whence the proof of the claim is concluded.
\end{proof}

We now amalgamate those facts together. According to (\ref{eqz: w* closure}), we may find a net $(f_\tau)_{\tau\in I}$ in ${\rm conv}\{\pm\p_\alpha\}_{\alpha<\omega_1}$ such that $f_\tau\to g$ in the $w^*$- topology; the elements $f_\tau$ have the form $f_\tau=\sum_{\alpha<\omega_1}d^\alpha_\tau\p_\alpha$, where, for each $\tau\in I$, $\sum_{\alpha<\omega_1}|d^\alpha_\tau|\leq1$ and only finitely many $d^\alpha_\tau$'s are different from zero.

The net $\left(\sum_{\alpha\in F}d^\alpha_\tau\p_\alpha\right)_{\tau\in I}$ is plainly a bounded net in a finite-dimensional Banach space (for what concerns the boundedness, observe that (\ref{eqz: equiv of basis in |||}) implies $\nn{\sum_{\alpha\in F}d^\alpha_\tau\p_\alpha}\leq1$). Hence, up to passing to a subnet, we may safely assume that it converges in $\nn{\cdot}$, and \emph{a fortiori} in the $w^*$-topology of $X^*$, to an element, say $\sum_{\alpha\in F}\tilde{d}_\alpha\p_\alpha$. The basis equivalence contained in the inequalities (\ref{eqz: equiv of basis in |||}) now allows us to deduce in particular that $\sum_{\alpha\in F}|\tilde{d}_\alpha|\leq1$.

As a consequence of this currently obtained norm convergence, we see that
$$\sum_{\alpha\notin F}d_\tau^\alpha\p_\alpha=f_\tau-\sum_{\alpha\in F}d_\tau^\alpha \p_\alpha\overset{w^*}{\longrightarrow}g-\sum_{\alpha\in F}\tilde{d}_\alpha\p_\alpha= \sum_{\alpha\notin F}g_\alpha\p_\alpha+\sum_{\alpha\in F}\left(g_\alpha-\tilde{d}_\alpha\right)\p_\alpha$$
and our present goal is to estimate the $\nn\cdot$-norm of the right hand side, making use of this $w^*$ convergence and the above claim.\smallskip

Let us fix temporarily $\eta>0$; from the $w^*$ convergence we obtain $\langle f_\tau,u\rangle\to\langle g,u\rangle=1$, whence the existence of $\tau_0\in I$ such that $\langle f_\tau,u\rangle\geq1-\eta$ for every $\tau\geq\tau_0$ \footnote{we are denoting by $\geq$ both the order on the real line and the preorder in the directed set $I$---this should cause no confusion, hopefully.}. Consequently, according to the above claim we deduce
$$\nn{\sum_{\alpha\notin F}d_\tau^\alpha\p_\alpha}\leq\sum_{\alpha\notin F} |d_\tau^\alpha|\leq12\eta\qquad(\tau\geq\tau_0).$$
This and the $w^*$ lower semi-continuity of the $\nn\cdot$-norm on $X^*$ then reassure us that
$$\nn{\sum_{\alpha\notin F}g_\alpha\p_\alpha+\sum_{\alpha\in F}\left(g_\alpha-\tilde{d}_\alpha\right)\p_\alpha}\leq12\eta.$$

Since $\eta>0$ was fixed arbitrarily, we may let $\eta\to0^+$ in the above inequality and, also exploiting the fact that $(\p_\alpha)_{\alpha<\omega_1}$ is a Schauder basis, we conclude that $g_\alpha=0$ whenever $\alpha\notin F$, while $g_\alpha=\tilde{d}_\alpha$ for every $\alpha\in F$. Therefore,
$$g=\sum_{\alpha<\omega_1}g_\alpha\p_\alpha=\sum_{\alpha\in F}g_\alpha\p_\alpha,$$
where $\sum_{\alpha\in F}|g_\alpha|=\sum_{\alpha\in F}|\tilde{d}_\alpha|\leq1$.
\end{proof}

Note that in the conclusion of the theorem we necessarily have $\sum_{\alpha\in F}|g_\alpha|=1$, since $1=\nn{g}\leq\sum_{\alpha\in F}|g_\alpha|$. Consequently, if $g=\sum_{\alpha\in F}g_\alpha\p_\alpha$ is any norm-attaining functional we have $\nn{\sum_{\alpha\in F}g_\alpha\p_\alpha}=\sum_{\alpha\in F}|g_\alpha|$. Since the set of norm-attaining functionals is dense in $X^*$, according to the Bishop--Phelps theorem, this equality holds true for every functional in $X^*$. We thus have the following immediate corollary.
\begin{corollary}\label{coroll: 1-equiv to ell1} $(\p_\alpha)_{\alpha<\omega_1}$ is isometrically equivalent to the canonical basis of $\ell_1(\omega_1)$.
\end{corollary}

In turn, this also implies (in the same notation as in the theorem)
$$\{\alpha<\omega_1\colon g_\alpha\neq0\}\subseteq\{\alpha<\omega_1\colon|\langle\p_\alpha,u\rangle|=1\}.$$
In fact, $1=\nn{g}=\sum_{\alpha\in F}|g_\alpha|$ and $1=\nn{u}\geq| \langle\p_\alpha,u\rangle|$ imply
$$1=\langle g,u\rangle=\sum_{\alpha<\omega_1}g_\alpha\langle\p_\alpha,u\rangle \leq\sum_{\alpha<\omega_1}|g_\alpha|\cdot|\langle\p_\alpha,u\rangle|\leq \sum_{\alpha<\omega_1}|g_\alpha|=1.$$
Consequently, all inequalities are, in fact, equalities and it follows that $|\langle\p_\alpha,u\rangle|=1$ whenever $g_\alpha\neq0$. It also follows that for every $\alpha<\omega_1$ we have $|g_\alpha|=g_\alpha\cdot\langle\p_\alpha,u\rangle$ and in particular $\langle\p_\alpha,u\rangle={\rm sgn}(g_\alpha)$ whenever $g_\alpha\neq0$.

As a piece of notation, it will be convenient to denote by ${\rm supp}(g)$ the set ${\rm supp}(g):=\{\alpha<\omega_1\colon g_\alpha\neq0\}$ for a functional $g=\sum_{\alpha<\omega_1}g_\alpha\p_\alpha\in X^*$.\smallskip

We can now approach the investigation of uncountable Auerbach systems in the space $X$. Our last observation for this first part is the fact that if $X$ contains an uncountable Auerbach system, then it also contains an uncountable Auerbach system such that the supports of the functionals are in a very specific position: either the supports are mutually disjoint and consecutive or the collection of the supports has an initial common root with cardinality $1$, followed by consecutive blocks. As it is to be expected, the $\Delta$-system lemma will play a prominent r\^ole in the proof; we shall also exploit again the same `transfer of mass' principle already present in the proof of the claim in Theorem \ref{norm-attain finitely supported}.
\begin{lemma}\label{"disjoint" supports} Assume that $X$ contains an uncountable Auerbach system. Then $X$ also contains an Auerbach system $\{\tilde{u}_\alpha;\tilde{g}_\alpha\}_{\alpha<\omega_1}$ such that one of the following two conditions is satisfied:
\begin{enumerate}
\item either there exists $\gamma<\omega_1$ with the properties that $\gamma=\min({\rm supp}(\tilde{g}_\alpha))$ for every $\alpha<\omega_1$ and ${\rm supp}(\tilde{g}_\alpha)\setminus\{\gamma\}<{\rm supp}(\tilde{g}_\beta)\setminus\{\gamma\}$ for every $\alpha<\beta<\omega_1$;
\item or ${\rm supp}(\tilde{g}_\alpha)<{\rm supp}(\tilde{g}_\beta)$ for every $\alpha<\beta<\omega_1$.
\end{enumerate}
\end{lemma}
Needless to say, the above conditions are mutually exclusive. Let us also point out explicitly that (1) implies in particular ${\rm supp}(\tilde{g}_\alpha)\cap{\rm supp}(\tilde{g}_\beta)=\{\gamma\}$ for $\alpha<\beta<\omega_1$, while (2) implies that the sets ${\rm supp}(\tilde{g}_\alpha)$ are mutually disjoint.

\begin{proof}\footnote{In this proof we shall denote by $\Delta$ a finite set, obtained via the $\Delta$-system lemma. This will not create conflicts with the parameter $\Delta\leq1/5$, fixed at the very beginning of the section, since such parameter will play no r\^ole here.} Assume that $X$ contains an uncountable Auerbach system and let us fix one, say $\{u_\alpha;g_\alpha\}_{\alpha<\omega_1}$; according to Theorem \ref{norm-attain finitely supported}, we know that the sets ${\rm supp}(g_\alpha)$ are finite sets. In view of the $\Delta$-system lemma we can therefore assume (up to passing to an uncountable subset of $\omega_1$ and relabeling) that there exists a finite set $\Delta\subseteq\omega_1$ such that ${\rm supp}(g_\alpha)\cap{\rm supp}(g_\beta)=\Delta$ for distinct $\alpha,\beta<\omega_1$. In the case that $\Delta$ is the empty set we have obtained an uncountable Auerbach system with mutually disjointly supported functionals and a simple transfinite induction argument yields the existence of an uncountable subcollection of functionals whose supports are consecutive. In this case, (2) holds and we are done.\smallskip

Alternatively, if $\Delta\neq\varnothing$ we first have to concentrate all the mass present in $\Delta$ in a single coordinate. Fixed any $\gamma\in\Delta$, we have $\gamma\in{\rm supp}(g_\alpha)$, whence $|\langle\p_\gamma,u_\alpha\rangle|=1$ for every $\alpha<\omega_1$. Consequently, we can pass to an uncountable subset of $\omega_1$ and assume that $\langle\p_\gamma,u_\alpha\rangle$ has the same sign for every $\alpha<\omega_1$. If we repeat the same procedure for all the finitely many $\gamma$'s in $\Delta$ we see that we can assume without loss of generality that there are signs $\e_\gamma=\pm1$ ($\gamma\in\Delta$) such that $\langle\p_\gamma,u_\alpha\rangle=\e_\gamma$ for each $\alpha<\omega_1$ and $\gamma\in\Delta$. Let us also fix arbitrarily an element $\overline{\gamma}\in\Delta$ and assume, for simplicity, that $\e_{\overline{\gamma}}=1$ (if this is not the case, we can just achieve it by replacing the Auerbach system $\{u_\alpha;g_\alpha\}_{\alpha<\omega_1}$ by $\{-u_\alpha;-g_\alpha\}_{\alpha<\omega_1}$). \smallskip

We are now in position to define the functionals $\tilde{g}_\alpha$ ($\alpha<\omega_1$). Assume that
$$g_\alpha=\sum_{\gamma<\omega_1}g_\alpha^\gamma\p_\gamma=\sum_{\gamma\in\Delta} g_\alpha^\gamma\p_\gamma+\sum_{\gamma\in\Delta^\complement}g_\alpha^\gamma\p_\gamma,$$
where $\sum_{\gamma<\omega_1}|g_\alpha^\gamma|=1$, for every $\alpha<\omega_1$. We may then define
$$\tilde{g}_\alpha=\sum_{\gamma\in\Delta}|g_\alpha^\gamma|\cdot\p_{\overline{\gamma}} +\sum_{\gamma\in\Delta^\complement}g_\alpha^\gamma\p_\gamma;$$
plainly, $\nn{\tilde{g}_\alpha}=1$ too (this follows from Corollary \ref{coroll: 1-equiv to ell1}). \smallskip

Finally, in order to evaluate $\langle\tilde{g}_\alpha,u_\beta\rangle$, note preliminarily that if $\gamma\in\Delta$ then $\gamma\in{\rm supp}(g_\alpha)$ for each $\alpha<\omega_1$; since $g_\alpha$ attains its norm at $u_\alpha$, the remarks preceding the present lemma imply that
$$|g_\alpha^\gamma|=g_\alpha^\gamma\langle\p_\gamma,u_\alpha\rangle= g_\alpha^\gamma\e_\gamma =g_\alpha^\gamma\langle\p_\gamma,u_\beta\rangle,$$
whenever $\alpha,\beta<\omega_1$. Consequently,
\begin{eqnarray*}
\langle\tilde{g}_\alpha,u_\beta\rangle &:=& \sum_{\gamma\in\Delta}|g_\alpha^\gamma| \cdot\langle\p_{\overline{\gamma}},u_\beta\rangle+\sum_{\gamma\in\Delta^\complement} g_\alpha^\gamma\langle\p_\gamma,u_\beta\rangle\\
&=& \sum_{\gamma\in\Delta}|g_\alpha^\gamma|+\sum_{\gamma\in\Delta^\complement} g_\alpha^\gamma\langle\p_\gamma,u_\beta\rangle\\
&=& \sum_{\gamma\in\Delta}g_\alpha^\gamma\langle\p_\gamma,u_\beta\rangle+ \sum_{\gamma\in\Delta^\complement} g_\alpha^\gamma\langle\p_\gamma,u_\beta\rangle\\
&=& \sum_{\gamma<\omega_1}g_\alpha^\gamma\langle\p_\gamma,u_\beta\rangle=\langle g_\alpha,u_\beta\rangle.
\end{eqnarray*}
It follows that $\{u_\alpha;\tilde{g}_\alpha\}_{\alpha<\omega_1}$ is also an Auerbach system and the mutual intersections of the supports of the functionals $\tilde{g}_\alpha$ all reduce to the singleton $\{\overline{\gamma}\}$. Finally, a transfinite induction argument analogous to the one needed above proves the existence of an Auerbach system satisfying (1).
\end{proof}

We can finally pass to the second part of our considerations and prove the main result of the section. Prior to this, the following remark is dedicated to the presentation, in a~simplified setting, of one of the main ingredients in the proof of the result, concerning the choice of the parameters $\lambda_\alpha$.
\begin{remark}Let $u\in c_0(\omega_1)$ be any non-zero vector and fix an ordinal $\gamma<\omega_1$ with ${\rm supp}(u)<\gamma$. We note that there are possible many choices of $\lambda_\gamma$ such that the corresponding functional $\p_\gamma$ satisfies $\langle\p_\gamma,u\rangle\neq0$. Since the definition of $\p_\gamma$ depends on the choice of the parameter $\lambda_\gamma$, we shall also denote by $\p_\gamma(\lambda)$ the functional obtained choosing $\lambda_\gamma=\lambda$. Let us then observe that the function

$$\lambda\mapsto\langle\p_\gamma(\lambda),u\rangle:=\sum_{\alpha<\gamma}u(\alpha) \langle\p_\gamma(\lambda),e_\alpha\rangle=\sum_{k=1}^{|\gamma|} u(\sigma_\gamma(k))\lambda^k$$
is expressed by a power series with bounded coefficients, not all of which equal zero. Therefore $\lambda\mapsto\langle\p_\gamma(\lambda),u\rangle$ is a nontrivial real-analytic function on $(-1,1)$ and, in view of the identity principle for real-analytic functions, it necessarily has finitely many zeros in $(0,\delta)$. This consideration allows us for many choices of a parameter $\lambda_\gamma$ such that $\langle\p_\gamma,u\rangle\neq0$. In the course of the proof of the result to follow, we shall need to exploit the same argument involving analyticity in a more complicated setting. 
\end{remark}

\begin{theorem}[CH] There exists a choice of the parameters $(\lambda_\alpha)_{\alpha<\omega_1}$ such that the corresponding space $X=(c_0(\omega_1),\nn\cdot)$ does not contain any uncountable Auerbach system.
\end{theorem}\begin{proof} Let us denote by $c_0(\alpha)$ ($\alpha<\omega_1$) the subspace of $c_0(\omega_1)$ consisting of all vectors $u\in c_0(\omega_1)$ such that ${\rm supp}(u)\subseteq \alpha$. Since every element of $c_0(\omega_1)$ is countably supported, we have $c_0(\omega_1)=\cup_{\alpha<\omega_1}c_0(\alpha)$; moreover, every  space $c_0(\alpha)$, for $\omega\leq\alpha<\omega_1$, is isometric to $c_0$, hence it has cardinality the continuum. As a consequence, $|c_0(\omega_1)|=\mathfrak{c}$ too. Assuming (CH), we may therefore well order the non-zero vectors of $c_0(\omega_1)$ in an $\omega_1$-sequence $(v_\alpha)_{\alpha<\omega_1}$, \emph{i.e.}, $c_0(\omega_1)\setminus\{0\}=\{v_\alpha\}_{\alpha<\omega_1}$.\smallskip

As in the previous remark, we shall occasionally denote by $\p_\gamma(\lambda_\gamma)$ the functional $\p_\gamma$, whenever it will be desirable to stress the dependence of $\p_\gamma$ on the parameter $\lambda_\gamma$. Those parameters $(\lambda_\alpha)_{\alpha<\omega_1}$ will be chosen as to satisfy the conclusion of the following claim.

\begin{claim*} It is possible to choose the parameters $(\lambda_\alpha)_{\alpha< \omega_1}$ in such a way that the following two assertions {\bf (A)} and {\bf (B)} are satisfied.
\begin{description}
\item[(A)] For every $N\in\N$, $N\geq2$, for every choice of ordinal numbers $\alpha_1, \dots,\alpha_N<\omega_1$ and $\beta_1,\dots,\beta_N<\omega_1$ with the properties that:
\begin{romanenumerate}
\item $\{v_{\alpha_1},\dots,v_{\alpha_N}\}$ is a linearly independent set;
\item $\beta_1<\beta_2<\dots<\beta_N$;
\item $\alpha_1,\dots,\alpha_N<\beta_2$;
\item ${\rm supp}(v_{\alpha_1}),\dots,{\rm supp}(v_{\alpha_N})<\beta_2$;
\item $\langle\p_{\beta_1},v_{\alpha_i}\rangle\neq0$ for every $i=1,\dots,N$;
\end{romanenumerate}
one has:
$$\det\left(\left(\left\langle\p_{\beta_i},v_{\alpha_j}\right\rangle\right)_{i,j=1}^N\right)\neq0.$$

\item[(B)] For every $N\in\N$, for every choice of ordinal numbers $\alpha_1, \dots,\alpha_N<\omega_1$ and $\beta_1,\dots,\beta_N<\omega_1$ with the properties that:
\begin{romanenumerate}
\item $\{v_{\alpha_1},\dots,v_{\alpha_N}\}$ is a linearly independent set;
\item $\beta_1<\beta_2<\dots<\beta_N$;
\item $\alpha_1,\dots,\alpha_N<\beta_1$;
\item ${\rm supp}(v_{\alpha_1}),\dots,{\rm supp}(v_{\alpha_N})<\beta_1$;
\end{romanenumerate}
one has:
$$\det\left(\left(\left\langle\p_{\beta_i},v_{\alpha_j}\right\rangle\right)_{i,j=1}^N\right)\neq0.$$
\end{description}
\end{claim*}
\begin{proof}[Proof of the claim.] We shall argue by transfinite induction on $\gamma:=\beta_N<\omega_1$, observing that for $\gamma=0$ both conditions {\bf (A)} and {\bf (B)} are trivially satisfied, while $\p_0=e^*_0$ regardless of the choice of $\lambda_0$. We may therefore assume by the transfinite induction assumption that, for a certain $\gamma<\omega_1$, we have already chosen parameters $(\lambda_\alpha)_{\alpha<\gamma}$ with the corresponding functionals $\p_\alpha=\p_\alpha(\lambda_\alpha)$ in such a way that conditions {\bf (A)} and {\bf (B)} are satisfied for every choice of the ordinals as above, subject to the condition $\beta_N<\gamma$. In order to verify the claim for $\gamma$, we only need to define $\lambda_\gamma$ and therefore $\p_\gamma$; we shall be considering the function $\lambda\mapsto\p_\gamma(\lambda)$ and show that a suitable choice of $\lambda=\lambda_\gamma$ is possible.\smallskip

We shall first focus on choosing $\lambda$ in such a way to achieve condition {\bf (A)}. Let us select arbitrarily $N\geq2$ and ordinal numbers $\alpha_1,\dots,\alpha_N$ and $\beta_1,\dots,\beta_N$ satisfying conditions (i)--(v) and such that $\beta_N=\gamma$; note that in particular the functionals $\p_{\beta_1},\dots,\p_{\beta_{N-1}}$ have already been defined, and we only need to choose a suitable parameter $\lambda$ in $\p_{\beta_N}(\lambda)=\p_\gamma(\lambda)$. Observe that, by using the Laplace expansion for the determinant on the last column,
$$\det\left(\left(\left\langle\p_{\beta_i},v_{\alpha_j}\right\rangle\right) _{i,j=1}^N\right)=\sum_{j=1}^N(-)^{N+j}\left\langle\p_{\beta_N}(\lambda),v_{\alpha_j} \right\rangle d_j=\left\langle\p_{\beta_N}(\lambda),\sum_{j=1}^N(-)^{N+j}d_j v_{\alpha_j}\right\rangle,$$
where $d_j$ is the determinant of the $(N-1)\times(N-1)$ matrix obtained removing the $j$-th row and the $N$-th column from the original matrix $\left(\langle\p_{\beta_i},v_{\alpha_j}\rangle\right) _{i,j=1}^N$. In the case that $N=2$, then actually $d_j=\langle\p_{\beta_1},v_{\alpha_{3-j}}\rangle$ ($j=1,2$), whence $d_j\neq0$, according to (v). On the other hand, if $N\geq3$, then $d_j$ is the determinant of the matrix obtained from the action of the functionals $\{\p_{\beta_1},\dots,\p_{\beta_{N-1}}\}$ on the set of vectors $\{v_{\alpha_1},\dots,v_{\alpha_N}\}\setminus\{v_{\alpha_j}\}$. Those vectors and functionals plainly satisfy conditions (i)--(v) with $\beta_{N-1}<\gamma$ instead of $\gamma$; thereby, $d_j\neq0$ follows from the transfinite induction assumption.\smallskip

Consequently, in each case we may conclude that $d_j\neq0$ ($j=1,\dots,N$) and condition (i) then forces the vector $\sum_{j=1}^N(-)^{N+j}d_j v_{\alpha_j}$ to be non-zero. According to the remark preceding the proof, we may now deduce that the function 
$$\lambda\mapsto\left\langle\p_{\beta_N}(\lambda),\sum_{j=1}^N(-)^{N+j}d_j v_{\alpha_j}\right\rangle$$
is a non-trivial real-analytic function on $(-1,1)$ and consequently it has only finitely many zeros on the set $(0,\delta)$. However, due to conditions (ii) and (iii), there are only countably many choices for $N\in\N$, $\alpha_1,\dots,\alpha_N$ and $\beta_1,\dots,\beta_N$ as in {\bf (A)} and satisfying $\beta_N=\gamma$. Therefore, there exists a countable set $\Lambda_{\bf (A)}\subseteq(0,\delta)$ such that all the determinants appearing in {\bf (A)} and with $\beta_N=\gamma$ are different from zero, for every choice of $\lambda\in(0,\delta)\setminus\Lambda_{\bf (A)}$. With any such choice, condition {\bf (A)} is verified for $\gamma$.\smallskip

A very similar consideration also applies to condition {\bf (B)}. The argument is even simpler, and we therefore omit the straightforward modifications, since we don't need to distinguish between the cases $N=2$ and $N\geq3$, the above argument for $N\geq3$ now being applicable to every $N\geq1$; incidentally, this is the reason why we do not need condition (v) in this case. Therefore, we obtain a countable subset $\Lambda_{\bf (B)}$ of $(0,\delta)$ such that condition {\bf (B)} is verified for $\gamma$, whenever $\lambda\in(0,\delta)\setminus\Lambda_{\bf (B)}$.\smallskip

Finally, choosing any $\lambda_\gamma\in(0,\delta)\setminus(\Lambda_{\bf (A)}\cup\Lambda_{\bf (B)})$ and such that $\lambda_\gamma\neq\lambda_\alpha$ for $\alpha<\gamma$ then provides us with a functional $\p_\gamma$ for which both assertions {\bf (A)} and {\bf (B)} are satisfied for $\gamma$; therefore, the transfinite induction step is complete and so is the proof of the claim. 
\end{proof}

Having the claim proved, we may choose parameters $(\lambda_\alpha)_{\alpha<\omega_1}$ satisfying conditions {\bf (A)} and {\bf (B)} above; we then denote by $X:=(c_0(\omega_1),\nn\cdot)$ the space obtained as described in the first part of the section, where the functionals $(\p_\alpha)_{\alpha<\omega_1}$ are obtained from the presently chosen sequence $(\lambda_\alpha)_{\alpha<\omega_1}$. We may now conclude the proof, by showing that the space $X$ does not contain any uncountable Auerbach system.\smallskip

Assume by contradiction that such systems do exist. Then we may find one, say $\{u_\alpha;g_\alpha\}_{\alpha<\omega_1}$, that satisfies the conclusion to Lemma \ref{"disjoint" supports}. Moreover, according to Theorem \ref{norm-attain finitely supported} the sets ${\rm supp}(g_\alpha)$ ($\alpha<\omega_1$) are finite sets; therefore, we can also assume that they all have the same finite cardinality, say $N$. On the other hand, we shall presently show that both the cases contained in the conclusion to Lemma \ref{"disjoint" supports} are in contradiction with conditions {\bf (A)} and {\bf (B)}; this ultimately leads us to the desired contradiction and concludes the proof.\smallskip

Firstly, we show that the validity of {\bf (A)} rules out the possibility (1) in Lemma \ref{"disjoint" supports} to hold true. Assume by contradiction that the supports of $(g_\alpha)_{\alpha<\omega_1}$ satisfy condition (1) and let us write ${\rm supp}(g_\alpha):=\{\beta_1^\alpha,\dots\beta_N^\alpha\}$, where $\beta_1^\alpha<\beta_2^\alpha<\dots<\beta_N^\alpha<\omega_1$. According to assumption (1), we have $\beta_1^\alpha=\beta_1^\eta$ and $\{\beta_2^\alpha,\dots\beta_N^\alpha\}<\{\beta_2^\eta,\dots\beta_N^\eta\}$ for $\alpha<\eta<\omega_1$; it follows in particular that $\sup_{\alpha<\omega_1}\beta_2^\alpha=\omega_1$. Moreover, the non-zero vectors $\{u_1,\dots,u_N\}$ have been enumerated in the $\omega_1$-sequence $(v_\alpha)_{\alpha<\omega_1}$, so we can find indices $\alpha_1,\dots,\alpha_N$ such that $u_j=v_{\alpha_j}$ for every $j=1,\dots,N$. We may also fix an ordinal number $\overline{\alpha}<\omega_1$ such that $\alpha_j<\overline{\alpha}$ and ${\rm supp}(v_{\alpha_j})<\overline{\alpha}$ for $j=1,\dots,N$. \smallskip

Since $\sup_{\alpha<\omega_1}\beta_2^\alpha=\omega_1$, it is possible to choose an ordinal $\alpha<\omega_1$ with the property that $\beta_2^\alpha>\overline{\alpha}$; needless to say, we may also assume that $\alpha>N$. Let us set $\beta_j:=\beta_j^\alpha$, for such a choice of $\alpha$. With such a choice of the indices $\{\alpha_1,\dots\alpha_N\}$ and $\{\beta_1,\dots \beta_N\}$ it is apparent that requirements (ii)--(iv) in condition {\bf (A)} are satisfied; also (i) is undoubtedly valid. Therefore, we only need to check the validity of (v): in order to achieve this, note preliminarily that $\beta_1\in{\rm supp}(g_\alpha)$ for every $\alpha<\omega_1$, in particular for $\alpha=\alpha_j$. The comments following Corollary \ref{coroll: 1-equiv to ell1} and the fact that $g_{\alpha_j}$ attains its norm at $v_{\alpha_j}$ then assure us that actually $|\langle\p_{\beta_1},v_{\alpha_j}\rangle|=1$.\smallskip

Consequently, all the assumptions in condition {\bf (A)} have been verified and the validity of {\bf (A)} leads us to the conclusion that
$$\det\left(\left(\left\langle\p_{\beta_i},v_{\alpha_j}\right\rangle\right) _{i,j=1}^N\right)\neq0.$$
On the other hand, we may clearly write $g_\alpha:=\sum_{i=1}^Nc_i\p_{\beta_i}$ for some choice of numbers $c_i$ ($i=1,\dots,N$)~with $\sum_{i=1}^N|c_i|=1$. Since $\alpha>N$ we have
$$0=\langle g_\alpha,u_j\rangle=\langle g_\alpha,v_{\alpha_j}\rangle=\sum_{i=1}^Nc_i \langle\p_{\beta_i},v_{\alpha_j}\rangle\qquad j=1,\dots,N.$$
In matrix form, the present equations read
$$\begin{pmatrix}\langle\p_{\beta_1},v_{\alpha_1}\rangle &\dots& \langle\p_{\beta_N} ,v_{\alpha_1}\rangle\\ \vdots && \vdots\\ \langle\p_{\beta_1},v_{\alpha_N}\rangle &\dots& \langle\p_{\beta_N},v_{\alpha_N}\rangle \end{pmatrix}\cdot \begin{pmatrix}c_1\\ \vdots\\c_N\end{pmatrix}=\begin{pmatrix}0\\ \vdots\\0\end{pmatrix};$$
this obviously contradicts $\det\left(\left(\left\langle\p_{\beta_i},v_{\alpha_j}\right\rangle\right) _{i,j=1}^N\right)\neq0$ and therefore clause (1) in Lemma \ref{"disjoint" supports} cannot occur.\smallskip

A very similar argument, which we only sketch, proves that (2) is in contradiction with {\bf (B)}. In fact, under the assumption of the validity of (2), and keeping the above notation for ${\rm supp}(g_\alpha):=\{\beta_1^\alpha,\dots\beta_N^\alpha\}$, we conclude that $\sup_{\alpha<\omega_1}\beta_1^\alpha=\omega_1$; we also select $\alpha_1,\dots,\alpha_N$ and $\overline{\alpha}$, proceeding in the same way as above. We are now in position to choose $\alpha<\omega_1$ such that $\beta_1^\alpha>\overline{\alpha}$ and $\alpha>N$ and define $\beta_j:=\beta^\alpha_j$, for such a choice of $\alpha$. Having chosen $\{\alpha_1,\dots\alpha_N\}$ and $\{\beta_1,\dots \beta_N\}$ in such a way, requirements (i)--(iv) in condition {\bf (B)} are satisfied. Therefore {\bf (B)} assures us that
$$\det\left(\left(\left\langle\p_{\beta_i},v_{\alpha_j}\right\rangle\right) _{i,j=1}^N\right)\neq0$$
and a contradiction follows from \emph{verbatim} the same argument as in the previous case.

Consequently, both clauses (1) and (2) in the conclusion to Lemma \ref{"disjoint" supports} fail to hold and the contrapositive form to Lemma \ref{"disjoint" supports} itself implies that $X$ contains no uncountable Auerbach system, thereby concluding the argument.
\end{proof}

\begin{remark} In the above claim one could have replaced condition {\bf (B)} above with the following condition {\bf (C)}.
\begin{description}
\item[(C)] For every choice of a vector $v_\alpha$ such that:
\begin{romanenumerate}
\item $\alpha<\gamma$;
\item ${\rm supp}(v_\alpha)<\gamma$;
\end{romanenumerate}
one has
$$\langle\p_\gamma,v_\alpha\rangle\neq0.$$
\end{description}
The main motivation why one may consider the present condition {\bf (C)} is clearly that {\bf (B)} is way more complicated than {\bf (C)}, even though {\bf (A)} and {\bf (B)} are actually quite similar. On the other hand, the drawback of {\bf (C)} is that {\bf (A)}$\&${\bf (B)} are exactly the conditions needed to face the two possible cases contained in the conclusion to Lemma \ref{"disjoint" supports}, while the argument exploiting {\bf (A)}$\&${\bf (C)} would be somewhat more involved.
\end{remark}

\end{document}